\documentclass[journal]{new-aiaa}
\usepackage[utf8]{inputenc}
\usepackage{textcomp}

\usepackage{graphicx}
\usepackage{amsmath}
\usepackage[version=4]{mhchem}
\usepackage{siunitx}
\usepackage{longtable,tabularx}
\usepackage{float}
\usepackage{amsfonts}
\usepackage{comment}
\usepackage{placeins}

\usepackage{amsthm}
\usepackage{bm}
\usepackage{hyperref}
\usepackage{xcolor}
\usepackage{algorithmic}
\usepackage{textcomp}
\usepackage{caption}
\usepackage{subcaption}
\usepackage{cleveref}

\usepackage{tikz}
\usetikzlibrary{positioning}
\usetikzlibrary{shadows}

\usepackage{nicefrac}

\DeclareMathOperator{\sgn}{sgn}

\DeclareMathOperator{\vol}{vol}

\DeclareMathOperator*{\argmin}{arg\,min}
\DeclareMathOperator{\diag}{diag}

\newcommand{\theset}[1]{\left\{ #1 \right\} }
\newcommand{\setst}{\;\vert\;}

\newcommand{\bx}{\mathbf{x}}
\newcommand{\by}{\mathbf{y}}
\newcommand{\bxi}[1][i]{\mathbf{x}_{#1}}
\newcommand{\bp}{\mathbf{p}}
\newcommand{\bc}{\mathbf{c}}
\newcommand{\bq}{\mathbf{q}}
\newcommand{\bu}{\mathbf{u}}
\newcommand{\bd}{\mathbf{d}}
\newcommand{\bPhi}{\mathbf{\Phi}}
\newcommand{\bzeta}{\bm{\zeta}}
\newcommand{\bsigma}{\bm{\sigma}}
\newcommand{\bolds}[1]{\mathbf{#1}}
\newcommand{\bpsi}{\bm{\psi}}
\newcommand{\bA}{\mathbf{A}}
\newcommand{\bB}{\mathbf{B}}

\newcommand{\bv}{\mathbf{v}}
\newcommand{\bF}{\mathbf{F}}

\newcommand{\ol}[1]{\overline{#1}}
\newcommand{\ul}[1]{\underline{#1}}

\newcommand{\R}{\mathbb{R}}

\newcommand{\Z}{\mathbb{Z}}
\newcommand{\calR}{\mathcal{R}}
\newcommand{\calU}{\mathcal{U}}

\newcommand{\calB}{\mathcal{B}}
\newcommand{\calD}{\mathcal{D}}

\newcommand{\calX}{\mathcal{X}}
\newcommand{\rd}{\mathop{}\!\mathrm{d}}

\newcommand{\pder}[2][]{\dfrac{\partial#1}{\partial#2}}
\newcommand{\Lp}[1][p]{\mathcal{L}^{#1}}

\newcommand{\Lpnorm}[2][p]{\| #2 \|_{\mathcal{L}^{#1}}}

\newcommand{\deltae}{\delta_{\textrm{e}}}
\newcommand{\deltaeo}{\delta_{\textrm{e}_0}}
\newcommand{\deltat}{\delta_{\textrm{th}}}
\newcommand{\deltato}{\delta_{\textrm{th}_0}}

\newcommand{\reachsets}[1][T]{\mathcal{R}(#1)}

\newtheorem{proposition}{Proposition}
\crefname{proposition}{Proposition}{Propositions}
\Crefname{proposition}{Proposition}{Propositions}

\crefname{theorem}{Theorem}{Theorems}
\Crefname{theorem}{Theorem}{Theorems}

\newtheorem{lemma}{Lemma}
\crefname{lemma}{Lemma}{Lemmas} 
\Crefname{lemma}{Lemma}{Lemmas}

\crefname{example}{Example}{Examples}

\crefname{enumi}{}{}     
\crefname{figure}{Fig.}{Figs.}

\newcolumntype{C}[1]{>{\hspace{0pt}\centering\arraybackslash}p{#1}}

\def\Ra#1{{#1}}
\def\Rb#1{{#1}}

\newcommand{\red}[1]{{#1}} 

\title{Reachability-Based Design Optimization for Aircraft Maneuverability}

\author{Steven Nguyen\footnote{Ph.D. Student, Department of Mechanical and Aerospace Engineering}, Nicholas Orndorff\footnote{PhD Candidate, Department of Mechanical and Aerospace Engineering, AIAA Student Member.}, Jorge Cort\'es\footnote{Professor, Department of Mechanical and Aerospace Engineering}, and Boris Kramer\footnote{Associate Professor, Department of Mechanical and Aerospace Engineering, and AIAA Senior Member.}}
\affil{University of California San Diego, La Jolla, CA, 92093}

\begin{document}

\maketitle

\begin{abstract}
This paper presents a method for incorporating control analysis into design optimization for highly-maneuverable aircraft.
By studying reachable sets for aircraft dynamics, we ensure that the optimizer will take the aircraft's controlled capabilities into account.
We compute reachable sets of linear dynamics for computational efficiency, and account for aircraft trim points to factor in asymmetric magnitude bounds on the input signals.
We demonstrate the proposed method in design optimization of a blended-wing-body aircraft.
Considering \Ra{its longitudinal dynamics, we optimize the aircraft based} on its reachable sets to \Rb{yield improvements in its controlled performance.
When designing a reference tracking controller, we find up to 57\% faster settling time with approximately 77\% less control effort for the optimized model's nonlinear dynamics.}
\end{abstract}

\section*{Nomenclature}


\Rb{{\renewcommand\arraystretch{1.0}
\noindent\begin{longtable*}{@{}l @{\quad=\quad} l@{}}
$\calX_0$                             & set of allowable initial conditions    \\
$\calU$                               & set of allowable control inputs \\
$\mathcal{D}$                         & set of design parameters \\
$f$                                   & design optimization objective function; $f:\mathcal{D} \to \R$   \\
$g_i$                                 & design optimization constraint functions, $i \in [1,k]$; $g_i:\mathcal{D} \to \R$ \\
$\bd$                                 & design parameters, $\bd \in \mathcal{D}$ \\
$\bA,\bB$                             & linear system matrices, $\bA \in \R^{n\times n}, \bB \in \R^{n\times m}$ \\
$\bx(t)$                              & state of dynamical system at time $t$, $\bx(t) \in \R^n$  \\
$\bu(t)$                              & input to dynamical system at time $t$, $\bu(t) \in \calU$  \\
$c$                                   & center half-span, m  \\
$w$                                   & wing half-span, m \\
$\ell$                                & transition half-span, m\\
$\Lambda$                             & wing sweep, rad\\
$\Gamma$                              & wing dihedral, rad\\
$D$                                   & induced drag, N   \\
$L$                                   & lift, N    \\
$F_{\textrm{th}}$                     & thrust force, N    \\
$M$                                   & pitching moment, Nm    \\
$m$                                   & mass, kg\\
$J_y$                                 & pitching moment of inertia, N$\textrm{m}^2$ \\
$\Z_+$                                & set of natural numbers excluding zero \\
\multicolumn{2}{@{}l}{Subscripts}\\
$0$ & trimmed values \\
th & thrust \\
e	& elevator \\
\end{longtable*}}}

\section{Introduction}

Traditional aircraft design separates the analyses of structural stability and controlled performance due to the daunting complexity of the design problem.
Whereas the forces and stresses an aircraft must withstand are well-studied and predictable, control-theoretic performance metrics could range from measuring an aircraft's capacity for disturbance rejection to its agility throughout the flight envelope.
Even when one aspect of the controlled performance is prioritized, designing a metric to quantify it often depends on assumptions about the model or controller architecture.
Given the flexibility of control-theoretic tools and the ease of software-level changes, the challenge of simultaneous control and design is often avoided by reserving control analysis for late stages of design, after the aircraft's physical features are established.
However, the design-control separation naturally comes with disadvantages and can lead to the design of less efficient systems~\cite{smith_OptimalMixPassive_1992} that are costly to resolve as the design pinballs between control and design teams.
Famously, NASA's Explorer 1 suffered from instabilities after launching due to unexpected interactions between its flexible structures and the control modes~\cite{deyst_SurveyStructuralFlexibility_1969}.
NASA also conducted a large study between numerous universities and aerospace companies to redesign the pointing control system of the Hubble telescope when it suffered problematic perturbations after launch~\cite{nurre_InitialPerformanceImprovements_1991,bukley_HubbleSpaceTelescope_1995}.
In the design of next-generation aircraft, control theory is critical to ensuring that new aircraft are more capable than their predecessors, and incorporating control-theoretic metrics into the early stages of design optimization could be paramount to unlocking performance gains in new aircraft.
Designing metrics that strike a balance between identifying key characteristics of flight behavior without restrictive assumptions on the control architecture is an open challenge whose answer varies greatly depending on the desired aircraft capabilities.

In this paper, we consider metrics based on reachability theory for control co-design of aircraft.
Leveraging controller-agnostic insights regarding the capabilities of aircraft from reachability theory, we propose optimization methods based on linear analysis that yield improvements in the controlled performance of the nonlinear dynamics. In Section~\ref{subsec:control_codesign}, we review the study of control co-design, which entails simultaneously solving problems of control and design of structures.
In Section~\ref{subsec:aircraft_design}, we survey the role that control theory has played historically in aircraft design.
In Section~\ref{subsec:bwb_review}, we review the history for blended-wing-body aircraft, a special type of aircraft on which we apply the proposed design optimization.
Lastly, in Section~\ref{subsec:reach_review}, we provide context on the study of reachability theory.
%
%

\subsection{Control Co-Design}\label{subsec:control_codesign}
The interconnections between control theory and the design of systems to be controlled have long been studied by engineers seeking to push the limits of performance.
Control co-design refers to the simultaneous consideration of the design of physical and control components, and has shown success in a number of applications~\cite{garcia-sanz_ControlCoDesignEngineering_2019}, such as design of vehicles~\cite{evangelousimos_ControlMotorcycleSteering_2006}, chemical reactors~\cite{luyben_AnalyzingInteractionDesign_1994,luyben_AnalyzingInteractionDesign_1994a}, and networked control systems~\cite{branicky_SchedulingFeedbackCodesign_2002}.
In structural engineering, the use of passive control to improve structural stability has a long history dating back to the introduction of tuned mass dampers~\cite{Frahm1911}, which form critical components of some of the world's tallest structures today~\cite{li_DynamicBehaviorTaipei_2011}.
Many studies of design with active control have shown that considering the effects of feedback can eliminate structural inefficiencies and reduce the effort required to stabilize structures~\cite{khot_MulticriteriaOptimizationDesign_1998}.
Similar results have been found for helicopter rotors designed simultaneously with their flight control systems~\cite{sahasrabudhe_IntegratedRotorFlightControl_1997,oktay_CombinedPassiveActive_2016} and flexible robotic arms designed for specific input-output responses~\cite{asada_ControlconfiguredFlexibleArm_1991}.
%
%
More broadly, the optimization of systems that connect multiple disciplines has been considered in the field of multidisciplinary design optimization (MDO) since the 1980s~\cite{khot_MulticriteriaOptimizationDesign_1998}.
Due to the coupling between structures, controls, and aerodynamics, aircraft systems have been popular subjects for MDO since its inception~\cite{livne_IntegratedMultidisciplinarySynthesis_1990,ashley_MakingThingsBestAeronautical_1982,grossman_IntegratedAerodynamicStructural_1988,grossman_IntegratedAerodynamicstructuralDesign_1990,wrenn_MultilevelDecompositionApproach_1988}.

\subsection{Control in Aircraft Design}\label{subsec:aircraft_design}
Despite the numerous investigations of control as a discipline within MDO for aicraft design, no standard modeling approach exists.
%
%
One idea is to parameterize a control scheme so that the design optimization chooses the best design simultaneously with the best controller -- this approach has been employed in MDO to optimize over airfoils while simultaneously optimizing over a predetermined control architecture~\cite{zeiler_IntegratedAeroservoelasticTailoring_1988,livne_IntegratedMultidisciplinarySynthesis_1990,khot_MulticriteriaOptimizationDesign_1998,yesilbas_UpgradingPerformanceSUAV_2026,oktay_CombinedPassiveActive_2016,sal_SimultaneousSweptAnhedral_2023}.
However, this approach can prove restrictive by focusing on one specific control scheme.
Deviations from that predetermined control synthesis technique would invalidate the optimization analysis.
On the other hand, a controller-agnostic design metric grants flexibility when designing control, but can be challenging to formulate due to the complexities of different approaches to nonlinear control.
Leveraging improvements in computational power, controller-agnostic approaches have become popular in recent years in combination  with robust and optimal control techniques for aircraft design~\cite{cunis_IntegratingNonlinearControllability_2023,gupta_ControllabilityGramianControl_2020,bahiamonteiro_DesignMetricsLanding_2024}.
Specifically, linearized aircraft dynamics are considered in~\cite{gupta_ControllabilityGramianControl_2020}, where the eigenvalues of the \red{controllability} Gramian are used as a measure of the aircraft's controllability.
By identifying the eigenvalues associated with desirable or non-desirable modes, design parameters that maximize and minimize those characteristics in the controlled aircraft dynamics are chosen.
Robustness metrics are considered in~\cite{bahiamonteiro_DesignMetricsLanding_2024}, where the Bode sensitivity integral is leveraged to capture the aircraft's capacity for disturbance rejection.
Although both of these metrics quantify desired performance characteristics, they assume that the aircraft will be linearly controlled.
In~\cite{cunis_IntegratingNonlinearControllability_2023}, nonlinear control design techniques are used by optimizing the aircraft according to a nonlinear optimal control problem that encodes a desired flight maneuver.
This approach precisely characterizes an aircraft's capabilites in performing desired maneuvers at the cost of increased computational expense and the knowledge of exact maneuvers to follow.
Building on these works in controller-agnostic control co-design, we propose using reachability theory to inform the design of highly-maneuverable aircraft.

\subsection{Blended-Wing-Body Aircraft}\label{subsec:bwb_review}
As a motivating platform, recent investigations into blended-wing-body (BWB) airplanes have exposed the importance of aircraft design techniques that can handle the complex coupling between different subsystems.
Blended-wing-body airplanes promise higher lift-to-drag ratios and lower empty weight compared to traditional tube-and-wing airplanes for equivalent missions.
This contributes to decreased fuel burn and lower operating costs on the order of 20-30\%~\cite{liebeck_DesignBlendedWing_2004}.
Boeing’s X48, later named BWB450, was one of the first BWB concepts that progressed from concept to sub-scale flight test, intending to compete against Airbus’s A380 with a purported 32\% reduction in fuel consumption for a 7,750 \red{nautical mile} route~\cite{larrimer2020beyond}.
However, these first experiments with BWB airplanes exposed serious challenges with control authority and control allocation, which were sometimes at odds with design choices which were made to optimize performance~\cite{larrimer2020beyond}.
This highlights the multidisciplinary nature of BWB aircraft, where the lack of distinct features (e.g., fuselage and wing) leads to complex and sometimes unintuitive tradeoffs in the design space, motivating the need for MDO with control-related metrics as a comprehensive tool for designing BWBs~\cite{wakayama_ChallengePromiseBlendedwingbody_1998}.
%
%
Aerodynamic shape optimization of the BWB450 and its variants has been performed~\cite{lyu_AerodynamicDesignOptimization_2014}, but these results consider only the most simple static stability constraints.
Recent commercial BWB variants~\cite{ahuja_ComparisonBlendedWing_2025} are targeting mid-size markets that compete with airplanes such as the Boeing 767-300 with around 200 passengers.
Many concepts are sized such that they can be multi-role, including tanker and cargo variants.

\subsection{Reachability Theory}\label{subsec:reach_review}

In control theory, the study of reachability considers the set of states that a dynamical system can reach starting from an initial set~\cite{bansal_HamiltonJacobiReachabilityBrief_2017}. The methods developed to compute reachable sets depend upon the type of system, constraints, and objective.
For linear systems, forward reachable sets are often computed by propagating the set of initial states through the dynamics~\cite{girard_EfficientComputationReachable_2006,althoff_SetPropagationTechniques_2021}.
In particular, some classes of convex subsets of Euclidean space, such as ellipsoids and zonotopes, yield particularly efficient computation of reachable sets due to closure properties under linear transformations~\cite{girard_ReachabilityUncertainLinear_2005,kurzhanski_EllipsoidalTechniquesReachability_2000,kurzhanskiy_EllipsoidalTechniquesReachability_2007}.
Computational packages such as SpaceEx~\cite{frehse_SpaceExScalableVerification_2011} and CORA~\cite{althoff_IntroductionCORA2015_2015} are designed specifically to handle these computations.
For nonlinear systems, these methods become less efficient due to the lack of closure guarantees, but some extensions exist for propagating zonotopic sets~\cite{althoff_ZonotopeBundlesEfficient_2011}.
For the general case of nonlinear dynamical systems, the problem of reachability is instead typically framed through backwards reachable sets, which require solving the Hamilton-Jacobi partial differential equation and suffers from the curse of dimensionality~\cite{mitchell_ApplicationLevelSet_2002,tomlin_SynthesizingControllersNonlinear_1998, mitchell_TimedependentHamiltonJacobiFormulation_2005}.
Alternatively, sampling-based approaches use simulated trajectories of the system to approximate the reachable set and can be applied to systems with black-box models.
This approach to reachability has been studied for the nonlinear setting~\cite{lew_ConvexHullsReachable_2025,pmlr-v155-lew21a,pmlr-v168-lew22a}, modeled as a disturbance affecting the system, and in the linear setting~\cite{villegaspico_ReachabilityAnalysisLinear_2018} for the purpose of finding component-wise bounds on the system dynamics.
In the aerospace field, reachability has previously been studied for improving autonomous landing systems and air traffic control by accurately modeling for unsafe flight trajectories~\cite{bayen_AircraftAutolanderSafety_2007,teo_FlightDemonstrationProvably_2005,bayen_DifferentialGameFormulation_2003a}.
However, reachability theory has not yet been incorporated into the design of aircraft.

\subsection{Contributions}

We propose incorporating reachable sets of linearized dynamics in the design optimization of aircraft for high maneuverability.
We restrict our analysis to reachable sets of linear systems because the curse of dimensionality leads to prohibitively expensive computation for nonlinear reachability, and we develop the metrics to be computationally cheap for incorporation into MDO.
Whereas previous works use the controllability Gramian~\cite{gupta_ControllabilityGramianControl_2020}, which characterizes the reachable set under bounded energy inputs~\cite{kalman_ContributionsTheoryOptimal_1960}, we consider all possible trajectories under bounded magnitude inputs.
The proposed optimization procedure is depicted in Figure~\ref{fig:mdo-flowchart}.
At each iteration, the reachability module receives design variables for the aircraft from the optimization solver, generates the linearized model, and computes the reachable set of the system.
Then, properties of the reachable set, such as its volume and projections, guide the design optimization problem to improve the aircraft's maneuverability.

\begin{figure}[t]
    \centering
    \begin{tikzpicture}[
        node distance=1.8cm,
        every node/.style={font=\small},
        block/.style={
            draw,
            rectangle,
            rounded corners,
            align=center,
            minimum width=4.2cm,
            text width=3.5cm,
            minimum height=1.0cm,
            line width=0.8pt
        },
        bigblock/.style={
            draw,
            dashed,
            fill=orange!10,
            rectangle,
            rounded corners,
            inner sep=8pt,
            line width=1.2pt
        },
        arrow/.style={->, line width=1pt},
        innerblock/.style={
        block,
        draw=black!80,
        fill=white,
        drop shadow={opacity=0.3}
        },
    ]

    \node[block] (mdo) {MDO optimizer};

    \node[bigblock, right= of mdo,
          minimum width=6.5cm,
          minimum height=10.5cm] (analysis) {};

    \node[innerblock, anchor=north] (aero)
        at ([yshift=-0.8cm]analysis.north)
        {Aero Analysis};

    \node[innerblock, below=of aero] (lin)
        {Linearized Model $\dot{\bx} = \bA(\bd) \bx + \bB(\bd) \bu$};

    \node[innerblock, below=of lin] (reach)
        {Computation of Reachable Set, $\reachsets$};

    \node[innerblock, below=0.8cm of reach] (props)
        {Evaluation of Reachable Set Properties};

    \draw[arrow] (aero) -- node[anchor=west, text width=3cm] {Force and moment derivatives} (lin);
    \draw[arrow] (lin) -- node[anchor=west, text width=3 cm] {$\bA(\bd),\bB(\bd)$ matrices, trimmed parameters} (reach);
    \draw[arrow] (reach) -- (props);

    \draw[arrow] (mdo.north) |-
        node[anchor=east,pos=0.3, text width=2.6cm] {Design variables, $\bd$}
        (aero.west);


    \draw[arrow] (props.west) -|
        node[pos=0.75, anchor=east, text width=2.5cm]
        {Metrics $f(\reachsets)$, $g_i(\reachsets,\bd)$ based on reachability theory}
        (mdo.south);

    \end{tikzpicture}
    \caption{Proposed optimization procedure incorporating reachability-based metrics.}
    \label{fig:mdo-flowchart}
\end{figure}

By analyzing reachable sets of linearized dynamics at various points in the flight envelope of an aircraft, we gain insight into the maneuverability of the aircraft in a \textit{controller-agnostic} manner.
We apply this analysis to the BWB aircraft outlined in Section~\ref{sec:bwb_aircraft}, using a panel method code to compute control and stability derivatives.
The optimization results produce \Rb{designs with improvements in controlled performance as discussed in Section~\ref{sec:control_analysis}.}
\Ra{Code to reproduce our results is available at \url{https://github.com/nguyensteven01/BWB_Design_Opt}}~\cite{nguyen_Nguyensteven01BWB_Design_OptV10_2026}.

This work builds on our previous conference submission~\cite{NguyenCortesKramer_Reachability_ACC} by (i) extending the analysis to reachable sets under asymmetrically bounded magnitude inputs, (ii) solving design optimization problems for a more complex \Rb{aircraft} geometry with more design variables, (iii) providing \Rb{more thorough analysis} of the improvements in \Rb{control of} the optimized aircraft designs, including \Rb{their} nonlinear \Rb{dynamics}, and \Ra{(iv) comparing our proposed optimization to existing literature}.

\section{Problem Setting}

We consider the problem of mission-informed design optimization and propose to incorporate reachability analysis as a measure of the aircraft's maneuverability.
Inspired by the use of the \red{controllability} Gramian~\cite{gupta_ControllabilityGramianControl_2020} and  nonlinear optimal control~\cite{cunis_IntegratingNonlinearControllability_2023} for aircraft design, we propose to employ reachable sets of the linearized dynamics.
For the design optimization problem, we consider a design vector $\bd \in \mathcal{D} \subseteq \R^{n_d}$, describing the aircraft design, where $\mathcal{D}$ is the set of allowed design choices.
We define a function $f:\mathcal{D} \to \R$ as the objective for maximization and constraints $g_1,\ldots,g_k:\mathcal{D} \to \R$, where evaluating $f$ or any $g_i$ may require reachability analysis.
Whereas controllability Gramian-based metrics consider the reachable sets of linear systems under bounded-energy inputs and optimal control problems consider specific maneuvers, we propose to consider the reachable sets of linearized dynamics under bounded-magnitude inputs, which capture the entire set of states that can be attained by a linear system subject to actuation limits.
%
%

In Section~\ref{subsec:reach_set_prelims}, we cover preliminary results and define the reachable sets we are interested in computing.
In Section~\ref{subsec:reach_set_metrics}, we introduce metrics based on reachable sets for use in design optimization problems.

\subsection{Reachable Sets of Linearized Dynamical Systems}\label{subsec:reach_set_prelims}

We consider parameterized linear dynamics of the form
\begin{equation}
    \dot{\bx}(t) = \bA(\bd)\bx(t) + \bB(\bd)\bu(t), \label{eq:basic_linear_system}
\end{equation}
where $\bx(t) \in \R^n$ is the state, $\bu(t) \in \calU(\bd) \subseteq \R^m$ the input, and $\bA(\bd)$, $\bB(\bd)$ have appropriate dimensions.
The dependency of the input set $\calU(\bd)$ and the matrices $\bA(\bd),\bB(\bd)$ on $\bd$ highlights the fact that the dynamics depend on the design parameters.
We omit this dependency when clear.
We consider initial conditions starting in a set $\calX_0 \subseteq \R^n$.
We denote by $\bx(t;\bxi[0],t_0,\bu(\cdot),\bd)$ the solution at time $t$ of~\eqref{eq:basic_linear_system} starting from $\bx(t_0) = \bxi[0] \in \calX_0$, driven by control signal $\bu(\cdot):[t_0,T] \to \calU(\bd)$.
The initial condition, initial time, control signal, and design parameters are omitted when clear, in which case we use $\bx(t)$ instead.
Fixing the terminal time $T > t_0$, we define the reachable set at time $t \in [t_0,T]$ as:
\begin{equation}
    \calR(t;\calX_0,\calU(\bd),\bd) = \theset{\bx(t;\bxi[0],t_0,\bu(\cdot),\bd) \setst \bxi[0] \in \calX_0,\, \bu(\tau) \in \calU(\bd) \, \forall \tau \in [t_0,t]}.
\end{equation}
Moving forward, we consider reachable sets with $\calX_0 = \theset{\bolds{0}}$.
We note that there is no loss of generality here because nonzero initial conditions simply shift the reachable set of a linear system by the vector $e^{\bA t}\bx(t_0)$.
Since we are interested in reachable sets of aircraft, we model their capabilities under actuator saturation by considering bounded magnitude inputs and set $\calU(\bd) = \theset{\bu \in \R^m \setst \ul{\bu}(\bd) \leq \bu \leq \ol{\bu}(\bd)}$, where $\ul{\bu}(\bd),\ol{\bu}(\bd) : \mathcal{D} \to \R^m$ and the inequalities are applied component-wise.
Note that the case where $\ul{\bu} = -\ol{\bu}$ and $\ol{\bu}_i = C$ for all $i = 1,\ldots,m$ and a constant $C \in \R_{\geq0}$ corresponds to input signals with bounded $\Lp[\infty]$ norms, where $\Lpnorm[\infty]{\bu(\cdot)} = \max\limits_{t \in [t_0,T],1 \leq i \leq m } |u_i(t)|$.
In this case, any input $\bu(\cdot):[t_0,T] \to \calU$ clearly satisfies $\Lpnorm[\infty]{\bu(\cdot)} \leq \ol{\bu}$.
%
%
For the remainder of the paper, when it is clear what the initial states, admissible inputs, and design variables are, we simply refer to the reachable set at time $t$ as $\reachsets[t]$.
The following result characterizes the control signals that determine the exposed points of these reachable sets.

\begin{lemma}\label{lem:saturated_controls_on_boundary}
    Consider the linear system~\eqref{eq:basic_linear_system} associated with a fixed design parameter $\bd$.
    Let $\calU = \theset{\bu \in \R^m \setst \ul{\bu} \leq \bu \leq \ol{\bu}}$ where $\ul{\bu},\ol{\bu} \in \R^m$ are constants.
    Then, for any exposed point $\bp \in \reachsets$, there exists a vector $\bc$ such that the control $\bu^*(t)$, with components defined as
    \begin{equation}\label{eq:fully_saturated_control}
        u_i^*(t) = \left\{ \begin{matrix}
        \ol{u}_i, & \sgn(\psi_i(t;\bc)) \geq 0 \\
        \ul{u}_i, & \sgn(\psi_i(t;\bc)) < 0,
        \end{matrix} \right.
    \end{equation}
    where $\bpsi(t;\bc) = \bc^\top e^{\bA (T-t)}\bB \in \R^m$, drives the system from $\bx(t_0) = \bolds{0}$ to $\bx(T) = \bp$.
\end{lemma}
\begin{proof}
Note that by definition of exposed points for convex sets, there exists a vector $\bzeta \in \R^n$ such that $\bzeta^\top \bp > \bzeta^\top \bq, \, \forall \bq \in \reachsets \backslash \theset{\bp}$.
Setting $\bc = \bzeta$, we first show that the components of $\bpsi(t;\bc)$ are zero only on a set of measure zero.
Without loss of generality, consider the single input case (so $\psi(t;\bc)$ is a scalar and $\bB$ is a column vector) and suppose there exists an interval $[t_1,t_2] \subset [t_0,T]$ such that $\psi(t;\bc) = 0\, \forall t \in [t_1,t_2]$.
Consider the two controllers
\begin{equation}
    u^{(1)}(t) = \left\{\begin{matrix}
    u^*(t), & t \in [t_0,T] \backslash [t_1,t_2] \\
    0, & t \in [t_1,t_2]
    \end{matrix} \right.   \qquad u^{(2)}(t) = \left\{\begin{matrix}
        u^*(t), & t \in [t_0,T] \backslash [t_1,t_2] \\
        \tilde{u}(t), & t \in [t_1,t_2],
        \end{matrix} \right.
\end{equation} where $\tilde{u}(t):[t_1,t_2] \to \calU$ is an arbitrary function.
Then, $\bx^{(1)} = \bx(T;u^{(1)})$ and $\bx^{(2)} = \bx(T;u^{(2)})$ satisfy $\bc^\top \bx^{(1)} = \bc^\top \bx^{(2)}$.
To see this, recall the solution of LTI systems:
\begin{align}
\bx^{(1)} &= \int\limits_{t_0}^{t_1} e^{\bA(T-t)}\bB u^{(1)}(t) \rd t + \int\limits_{t_1}^{t_2} 0 \rd t + \int\limits_{t_2}^T e^{\bA(T-t)}\bB u^{(1)}(t) \rd t \\
\bx^{(2)} &= \int\limits_{t_0}^{t_1} e^{\bA(T-t)}\bB u^{(2)}(t) \rd t + \int\limits_{t_1}^{t_2} e^{\bA(T-t)}\bB \tilde{u}(t) \rd t + \int\limits_{t_2}^T e^{\bA(T-t)}\bB u^{(2)}(t) \rd t
\end{align}
Due to the assumption that $\psi(t;\bc) = \bc^\top e^{\bA(T-t)}\bB = 0$ on $[t_1,t_2]$, we find that
\begin{equation}
    \bc^\top \bx^{(1)} = \int\limits_{t_0}^{t_1} \bc^\top e^{\bA(T-t)}\bB u^{(1)}(t)\rd t + \int\limits_{t_2}^T \bc^\top e^{\bA(T-t)}\bB u^{(1)}(t)\rd t = \bc^\top \bx^{(2)}.
\end{equation}
However, $\bx^{(1)} \neq \bx^{(2)}$.
To see this, assume $\bx^{(1)} = \bx^{(2)}$.
This assumption implies that
\begin{equation*}
    \int\limits_{t_1}^{t_2} e^{\bA(T-t)}\bB \tilde{u}(t) \rd t = 0
\end{equation*} for any arbitrary $\tilde{u}(t)$.
This can only be the case if $\bB$ is in the null space of $e^{\bA(T-t)}$ for all $t \in [t_1,t_2]$.
However, the matrix exponential is always full rank, so by contradiction, we have that $\bx^{(1)} \neq \bx^{(2)}$.
This further contradicts the assumption that there was a unique point $\bp$ maximizing $\bc^\top \bp$, so by contradiction, we also have that $\psi(t;\bc)$ cannot be zero on a set of nonzero measure.
Finally, to see that $\bu^*(t)$ steers the system to a state that maximizes the linear functional of $\bc$, note that for any arbitrary controller $\bu(t)$,
\begin{align*}
\bc^\top \bx(T; \bu(t)) &= \int\limits_{t_0}^T \bc^\top e^{\bA(T-t)}\bB \bu(t) \rd t \\
&= \int\limits_{t_0}^T \sum\limits_{i=1}^m \psi_i(t) u_i(t) \rd t.
\end{align*}
Clearly, $\psi_i(t) u_i(t) \leq \psi_i(t) u^*_i(t),\, \forall t \in [t_0,T]$, and $\bp = \bx(T;\bu^*)$.
\end{proof}

Lemma~\ref{lem:saturated_controls_on_boundary} extends the results of~\cite{pecsvaradi_ReachableSetsLinear_1971} to the case where $\ol{\bu} \neq -\ul{\bu}$ and applies~\cite[Thm. 9]{pontryagin_MathematicalTheoryOptimal_2018} to the computation of reachable sets.
Lemma~\ref{lem:saturated_controls_on_boundary} clarifies that the set of exposed points of $\calR(t)$ is composed of states driven by bang-bang controls, which are control signals that are saturated at all times.
In combination with the standard result that reachable sets of linear systems with convex initial sets and control sets are themselves convex, Lemma~\ref{lem:saturated_controls_on_boundary} allows us to approximate $\calR(T)$ using the convex hull of its exposed points.
To do this, we sample $k$ values $\theset{\bc^{(1)}, \bc^{(2)}, \ldots, \bc^{(k)}}\subseteq \R^n $ for $k \in \Z_+$.
%
%
Then, each $\bc^{(i)}$ corresponds to a controller $\bu^{(i)}(\cdot):[t_0,T] \to \calU$ that optimally steers the system from the origin at time $t_0$ to a point $\bx^{(i)}$ at time $T$.
By definition, each point $\bx^{(i)} \in \reachsets$.
Finally, we take the convex hull of the set $\theset{\bx^{(1)},\bx^{(2)},\ldots,\bx^{(k)}} \subseteq \R^n$ as our approximation of the reachable set $\reachsets$.
This approach is similar to the reachable set computation method proposed in~\cite{lew_ConvexHullsReachable_2025} for the case of linear systems, but does not require assuming full control authority.
It is also similar to the method proposed in~\cite{villegaspico_ReachabilityAnalysisLinear_2018}, but extends beyond computing component-wise state bounds to approximating the reachable set.

\subsection{Reachable Set Metrics}\label{subsec:reach_set_metrics}
Once the reachable set of the linearized dynamics is available, one can consider various metrics  to inform design.
One candidate metric is the volume of the reachable set, which gives a holistic view of the maneuverability of the aircraft.
Placing no emphasis on maneuverability in any axis, optimizing for the volume of the reachable set simply enlarges the set of states that can be reached in finite time.
We compute that volume using Scipy's implementation of QHull~\cite{barber_QuickhullAlgorithmConvex_1996}, which uses Delaunay triangulation.

Another metric we consider is the extent of the reachable set in an axis of interest, which improves how quickly the aircraft can move along that axis.
For example, maximizing the set along the direction of $ [0, \ 1, \ 0, \ 0]^\top$ ensures that the aircraft can achieve a greater change in its angle of attack (AOA) for the same input signal.
For a unit vector $\bv$ along the chosen axis, we measure this by taking the Euclidean length of the orthogonal projection of the reachable set onto $\bv$.
We consider these metrics
%
%
as both objective functions to maximize and constraints to enforce in Section~\ref{sec:optimization_problem_overview}.
Although we do not consider it in this paper, another candidate metric may be to check for the inclusion of specific points or trajectories in the reachable set, as done in~\cite{cunis_IntegratingNonlinearControllability_2023}.

%
%

\begin{figure}[ht]
    \centering
    \includegraphics[width=0.52\linewidth]{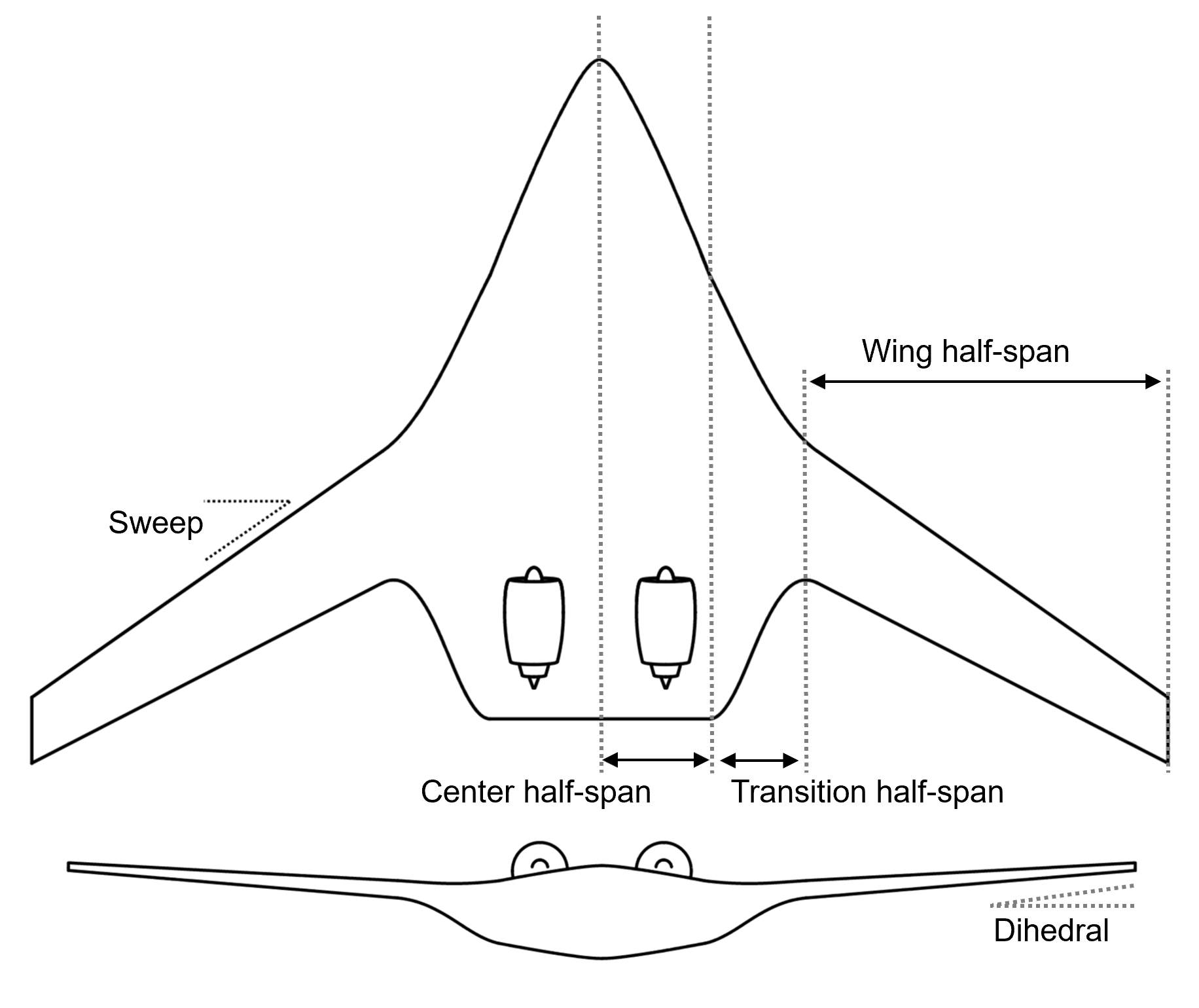}
    \caption{Blended-wing-body airplane concept used herein.
    The elevator control surface and engine are located at the rear of the BWB, and we consider the center half-span $c$, wing half-span $w$, transition half-span $\ell$, wing sweep $\Lambda$, and wing dihedral $\Gamma$ as design variables in the optimization problems herein.}
    \label{fig:bwb_diagram}
\end{figure}

\section{Blended-Wing-Body Aircraft}\label{sec:bwb_aircraft}
We consider a notional BWB that is sized similarly to the smaller 747 variants with approximately 160,000 lb payload capacity and a maximum takeoff weight of 350,000 lb.
The airplane is powered by two GE CF6-80A turbofans, each of which is assumed to provide 48,000 lb maximum thrust.
We only consider longitudinal control, which is provided by distinct elevator control surfaces at the rear of the center-body section, shown in Figure~\ref{fig:bwb_diagram}.
The elevator surfaces deflect simultaneously, acting effectively as a single control surface.
This provides greater longitudinal control authority, and allows the airplane to trim without requiring significant washout in the outer wing sections.
The elevator surfaces are deflected smoothly during simulation using free-form deformations, where a simplified bounding volume maps a commanded elevator deflection directly to the B-spline coefficients of the BWB center body.
The BWB’s outer mold line geometry is designed in NASA’s OpenVSP software~\cite{mcdonald2022open}. 

In Section~\ref{subsec:linear_model}, we provide the longitudinal equations of motion used for modeling and reachable set computations.
In Section~\ref{subsec:aero_data_collection}, we describe an interpolatory approach for querying the panel method results during optimization.

\subsection{Linearized Longitudinal Model}\label{subsec:linear_model}

Using airspeed $V$ in meters per second, angle of attack $\alpha$ in radians, pitch rate $Q$ in radians per second, and pitch angle $\theta$ in radians as states, we consider the longitudinal model from~\cite{stevens_AircraftControlSimulation_2015}:
\begin{align} \nonumber
    \dot{V}(t) &= \frac{F_{\textrm{th}}}{m} \cos(\alpha(t)) -\frac{D}{m}-g \sin(\theta(t)-\alpha(t)) \\ \label{eq:nonlinear_longitudinal_model}
    \dot{\alpha}(t) &= -\frac{F_{\textrm{th}}}{m V(t)} \sin(\alpha(t)) - \frac{L}{m V(t)} + \frac{g}{V(t)} \cos(\theta(t)-\alpha(t)) + Q(t) \\ \nonumber
    \dot{Q}(t) &= \frac{M}{J_y} \\ \nonumber
    \dot{\theta}(t) &= Q(t),
\end{align}
where the dependence of the scalars $F_{\textrm{th}},D,L$, and $M$ on the aircraft flight conditions are omitted.

We formulate the reachability analysis for linearized aircraft dynamics about the trim point, $[V_0,\ \alpha_0,\ Q_0,\ \theta_0]^\top$, with trimmed thrust and elevator inputs $[\deltato,\ \deltaeo]^\top$~\cite{stevens_AircraftControlSimulation_2015}\cite{lavretsky_RobustAdaptiveControl_2024}:
%
%
\begin{equation}\label{eq:lavret_wise_linear_long_model}
\begin{bmatrix}
\dot{V}(t) \\
\dot{\alpha}(t) \\
\dot{Q}(t) \\
\dot{\theta}(t)
\end{bmatrix}
=
\begin{bmatrix}
X_V & X_\alpha & 0 & -g \cos \gamma_0 \\
\frac{Z_V}{V_0} & \frac{Z_\alpha}{V_0} & 1 + \frac{Z_Q}{V_0} & -\frac{g \sin \gamma_0}{V_0} \\
M_V & M_\alpha & M_Q & 0 \\
0 & 0 & 1 & 0
\end{bmatrix}
\begin{bmatrix}
V(t) \\
\alpha(t) \\
Q(t) \\
\theta(t)
\end{bmatrix} +
\begin{bmatrix}
X_{\deltat} \cos \alpha_0 & X_{\deltae} \\
- X_{\deltat} \sin \alpha_0 & \frac{Z_{\deltae}}{V_0} \\
M_{\deltat} & M_{\deltae} \\
0 & 0
\end{bmatrix}
\begin{bmatrix}
\deltat(t) \\
\deltae(t)
\end{bmatrix},
\end{equation}
where $\delta_\textrm{th} \in [0,\ 1]$ is the linearized thrust throttle and $\deltae \in [-30,\ 30]$ deg is the linearized elevator rotation.
All other terms with subscripts are stability derivatives, as discussed in~\cite{stevens_AircraftControlSimulation_2015}, where $X,Z$, and $M$ refer to the $x$-axis forces, $z$-axis forces, and pitching moment.
Since linearized models are based on perturbations of the nonlinear model, the trimmed inputs are also considered when computing input bounds.

The aircraft is linearized around operating conditions of $V_0 = 200$ m/s, $Q_0 = 0$ \red{deg}/s, and a flight path angle of $\gamma_0 = \alpha_0 - \theta_0 = 0$ \red{deg}.
During design optimization, as the BWB geometry changes, we recompute the trimmed values of $\alpha_0,\, \theta_0,\, \deltato,\, \deltaeo$.
The values of the aerodynamic force and moment derivatives are computed using a panel method code, as discussed in Section~\ref{subsec:aero_data_collection}.
Due to constraints in the solver, derivatives with respect to the pitch rate are not computed, so $Z_Q=M_Q = 0$.

\subsection{Aerodynamic Data Collection}\label{subsec:aero_data_collection}
Aerodynamic modeling and initial MDO results for this BWB concept can be found in~\cite{scotzniovsky2025fast}, which \red{uses} a fully-differentiable panel method to perform a multi-point optimization consisting of several cruise conditions, climb, and structural sizing.
Rather than evaluating the panel method code at every iteration of the optimization problem, the aerodynamic data for the BWB is pre-computed offline and interpolated during optimization.
We collect aerodynamic data sweeping over \Ra{eight} input parameters.
These inputs sweep over airspeeds $V \in [100,\ 295]$ m/s, angle of attack \red{$\alpha \in [-5,\ 15]$ deg}, center half-span $c \in [3,\  7]$ m, wing half-span $w \in [10,\ 20]$ m, elevator rotation \red{$\deltae \in [-30,\ 30]$ deg}, \Ra{transition half-span $\ell \in [3,\ 8]$ m, wing sweep $\Lambda \in [15,\ 50]$ deg, and wing dihedral $\Gamma \in [10,\ 12]$ deg.}
The outputs consist of longitudinal forces in $x$- and $z$-axes, pitching moment, lift, drag, and all aerodynamic derivatives required for the linearized dynamics (except for $Z_Q$ and $M_Q$)~\eqref{eq:lavret_wise_linear_long_model}.
The offline computation is motivated by the fact that the derivatives are expensive to compute.
With \Ra{eight} swept parameters, collecting data with a resolution of \Ra{five} for each variable \Ra{except wing dihedral, which had a resolution of two}, requires 156,250 data points.
    
An example of the aerodynamic data is shown in Figure~\ref{fig:aero_sample}, where variation of the aerodynamic force in the $x$-direction, $F_x$, is shown for sweeps across the aircraft pitch and elevator angles.
The $x$-axis is taken parallel to the wind in this data.
A clear trend can be observed in Figure~\ref{fig:aero_sample}, where the magnitude of $F_x$ increases as the aircraft pitches up and rotates the elevator surface down, and is minimized when both pitch and elevator angle are approximately zero.
This is intuitive because drag forces will increase as the aircraft maximizes its surface area perpendicular to the wind and decrease when the aircraft faces the wind head-on.
A visualization of the BWB geometry with the forces acting on it from the panel method code is shown in Figure~\ref{fig:BWB_forces}.

\begin{figure}
\centering
\begin{subfigure}[t]{0.49\textwidth}
    \centering
    \includegraphics[width=\textwidth]{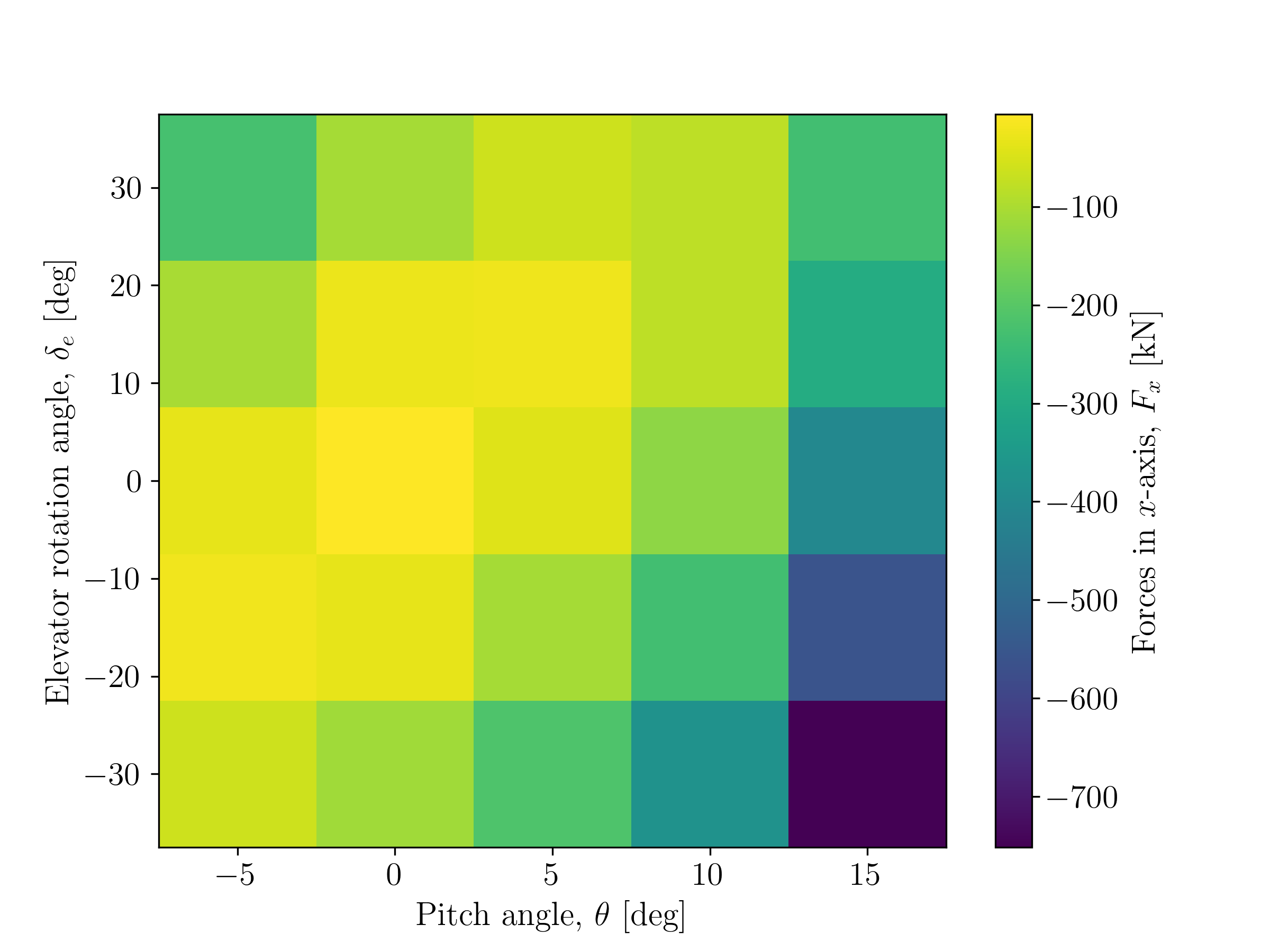}
    \caption{\Ra{Sample of aerodynamic data that shows variation in $x$-axis forces while sweeping over aircraft pitch and elevator angle.
    Other parameters are kept constant in this data.}}
    \label{fig:aero_sample}
\end{subfigure}
\hfill
\begin{subfigure}[t]{0.49\textwidth}
    \centering
    \includegraphics[width=\textwidth]{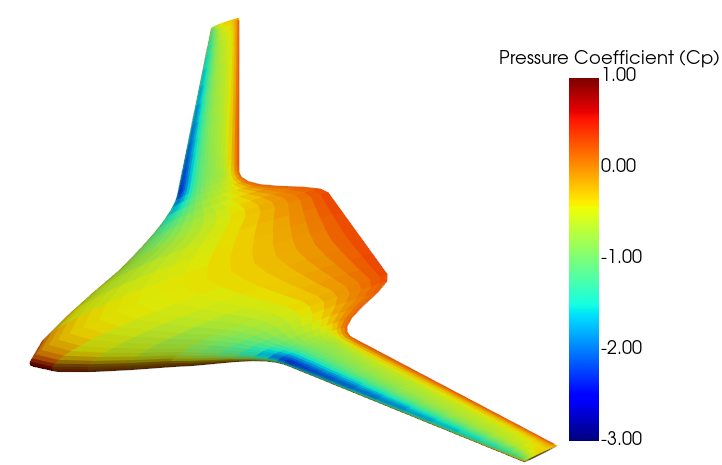}
    \caption{Sample visualization of the panel method computation of coefficient of pressure for the BWB.}
    \label{fig:BWB_forces}
\end{subfigure}
\caption{Outputs from panel method code for BWB aircraft.}
\label{fig:aero_model}
\end{figure}



\section{Optimization Problems}\label{sec:optimization_problem_overview}
We consider three optimization problems, outlined in Section~\ref{subsec:des_opt_problem}, for the BWB geometry using the reachability metrics discussed in Section~\ref{subsec:reach_set_metrics}.
The solutions to these problems are discussed in Section~\ref{subsec:optimization_results}.
\Ra{We consider the design variables $\bd = [c,w,\Lambda,\ell,\Gamma]^\top$ where $c$ is the center half-span, $w$ is the wing half-span, $\ell$ is the transition half-span, $\Lambda$ is the wing sweep angle, and $\Gamma$ is the wing dihedral of the BWB, as labeled in Figure~\ref{fig:bwb_diagram}.}
Since we collect aerodynamic data a priori, we set \red{$\calD$ to reflect the variable bounds as discussed in Section~\ref{subsec:aero_data_collection}.}
We use initial values \Ra{$\bd_0^\top = [c_0,w_0,\ell_0,\Lambda_0,\Gamma_0]^\top = [5\textrm{ m},12\textrm{ m}, 4.9\textrm{ m}, 35\textrm{ deg}, 10\textrm{ deg}]^\top$} for all optimization problems.
These initial values correspond with the original, undeformed outer mold line, which is sized according to the semi-empirical planform sizing equations for transport-category aircraft~\cite{raymer_AircraftDesignConceptual_2024}.
As outlined in Section~\ref{subsec:reach_set_prelims}, the reachable set is defined with initial set $\calX_0 = \theset{\bolds{0}}$.
At each step of the iteration, given design $\bd$, the aircraft is trimmed with an airspeed of $V_0 = 200$ m/s, pitch rate of $Q_0 = 0$ \red{deg}/s, and climb angle of $\gamma_0 = 0$ \red{deg}.
The resulting trimmed values of angle of attack $\alpha_0$, pitch $\theta_0$, elevator $\deltaeo$, and thrust $\deltato$ are used to compute the linear model matrices $\bA(\bd)$ and $\bB(\bd)$ as well as the input bounds $\calU(\bd) = [-\deltato,1-\deltato]\times [-30 - \deltaeo, 30 - \deltaeo]$.
This choice of allowable inputs ensures the controls for thrust and throttle stay within their respective ranges once the trimmed inputs are factored in.

\subsection{Design Optimization Formulations}\label{subsec:des_opt_problem}

The first optimization problem (cf. Section~\ref{subsubsec:volmax}) is to maximize the volume of the reachable set, denoted by $\vol(\reachsets[T;\bd])$.
The second optimization problem (cf. Section~\ref{subsubsec:dirmax}) is to maximize the length of the reachable set in a predetermined axis of interest~$\bv$, denoted by $\textrm{Proj}_\bv(\reachsets[T;\bd])$.
The third optimization problem  (cf. Section~\ref{subsubsec:volmax_dircon}) combines the first two by maximizing the volume of the reachable set while imposing that the length of the reachable set in the specified direction is increased.
We refer to the solutions of these three optimization problems as the \emph{volume-maximized}, \emph{directionally-maximized}, and \emph{volume-maximized, directionally-constrained} designs, respectively.
\red{In each problem, we also impose that the static margin of the aircraft, computed as $M_\alpha / L_\alpha$, remains positive to ensure longitudinal pitching stability~\cite{stevens_AircraftControlSimulation_2015}.}

\subsubsection{Volume-Maximized Optimization}\label{subsubsec:volmax}
The first optimization problem we consider is volume-maximization: \begin{align}\label{eq:volume_max_optimization} \tag{VM}
        \max_{c,w, \Ra{\ell, \Lambda, \Gamma}} \quad & \vol(\reachsets[T;{[c,\ w,\ \Ra{\ell,\ \Lambda,\ \Gamma}]^\top}] )\\ \nonumber
        \textrm{s.t.} \quad  
        & 3 \leq c \leq 7 \\ \nonumber 
        & 10 \leq w \leq 20\\ \nonumber
        & \Ra{3 \leq \ell \leq 8} \\ \nonumber
        & \Ra{15 \leq \Lambda \leq 50 }\\ \nonumber
        & \Ra{10 \leq \Gamma \leq 12} \\ \nonumber
        & \red{\textrm{Static Margin > 0}},
\end{align}
%
%
%
%
Taking the volume of the reachable set serves as a general notion of maneuverability for the aircraft, as a larger set of reachable states implies that the aircraft can maneuver around a larger swath of the state space.
However, since the volume gives no priority to any state, this approach can lead to designs that have more limited maneuverability in specific axes after optimization.

\subsubsection{Direction-Maximization Optimization}\label{subsubsec:dirmax}
In cases where the engineer knows that some states are more important for the aircraft's desired maneuver at the given trim point, then using the volume of the reachable set may be too general of a consideration in the optimization problem.
Instead, the second optimization problem we consider, directional-maximization, maximizes the length of the reachable set in a pre-specified direction $\bv$:
\begin{align}\label{eq:dirmax_optimization} \tag{DM}
        \max_{c,w, \Ra{\ell, \Lambda, \Gamma}} \quad & \| \textrm{Proj}_\bv \reachsets[T;{[c, \ w,\ \Ra{\ell,\ \Lambda,\ \Gamma}]^\top}] \|_2 \\ \nonumber
        \textrm{s.t.} \quad  
        & 3 \leq c \leq 7 \\ \nonumber 
        & 10 \leq w \leq 20\\ \nonumber
        & \Ra{3 \leq \ell \leq 8} \\ \nonumber
        & \Ra{15 \leq \Lambda \leq 50 }\\ \nonumber
        & \Ra{10 \leq \Gamma \leq 12}\\ \nonumber
        & \red{\textrm{Static Margin > 0}},
\end{align} %
Solving~\eqref{eq:dirmax_optimization} chooses $\bd$ such that the projection of $\reachsets[T;\bd]$ onto $\bv$ is as large as possible.

\subsubsection{Volume-Maximized, Directionally-Constrained Optimization}\label{subsubsec:volmax_dircon}
The third optimization problem we consider, volume-maximized and directionally-constrained, blends optimization problems~\eqref{eq:volume_max_optimization} and~\eqref{eq:dirmax_optimization} by maximizing the volume of the reachable set, but imposing a constraint on the projection of the reachable set:
\begin{align}\label{eq:volmax_dircon_optimization} \tag{VMDC}
        \max_{c,w, \Ra{\ell, \Lambda, \Gamma}} \quad & \vol (\reachsets[T;{[c, \ w,\ \Ra{\ell,\ \Lambda,\ \Gamma}]^\top}] )\\ \nonumber
        \textrm{s.t.} \quad  
        & \| \textrm{Proj}_\bv \reachsets[T;{[c, \ w,\ \Ra{\ell,\ \Lambda,\ \Gamma}]^\top}] \|_2 \geq (1+\kappa) \|\textrm{Proj}_\bv \reachsets[T;{[c_0, \ w_0,\ \Ra{\ell_0,\ \Lambda_0,\ \Gamma_0}]^\top}] \|_2 \\ \nonumber 
        & 3 \leq c \leq 7 \\ \nonumber 
        & 10 \leq w \leq 20\\ \nonumber
        & \Ra{3 \leq \ell \leq 8} \\ \nonumber
        & \Ra{15 \leq \Lambda \leq 50 }\\ \nonumber
        & \Ra{10 \leq \Gamma \leq 12}\\ \nonumber
        & \red{\textrm{Static Margin > 0}},
\end{align}
where $\kappa \in (0,1]$ ensures an enlargement of the reachable set in the desired projection direction.
Taking this approach balances between~\eqref{eq:volume_max_optimization} and~\eqref{eq:dirmax_optimization} by finding values for the design variables that enlarge the reachable set without compromising along axis $\bv$.

\subsection{Optimization Results}\label{subsec:optimization_results}
We solve problems~\eqref{eq:volume_max_optimization},~\eqref{eq:dirmax_optimization}, and~\eqref{eq:volmax_dircon_optimization} using sequential quadratic programming via PySLSQP~\cite{joshy_PySLSQPTransparentPython_2024}.
Although this is a gradient-based approach, we do not provide an analytic expression for the gradient of the reachable set metrics due to the complexity of computing reachable sets.
%
%

However, we note that it is reasonable to expect the gradients to be well-defined.
This is because the reachable set metrics require computing trimmed states from the design parameters, approximating linear dynamics matrices $\bA(\bd),\bB(\bd)$ from the trimmed states, and sampling solutions of linear systems by computing $\bx(T) = \int^T_0 e^{\bA(d)(T-t)} \bB(\bd) \bu(t) \rd t$.
Sampling the solutions of linear systems involves composing a matrix exponential, product, and integral, which are continuously differentiable operations.
We approximate the system matrices from the trimmed states by linearly interpolating the aerodynamic data collected from the panel method code, as described in Section~\ref{subsec:aero_data_collection}, which is continuously differentiable everywhere outside of the sampled points themselves, which form a set of measure zero.
Lastly, the mapping from design variables to trimmed states is locally continuously differentiable, which can be shown using the implicit function theorem~\cite{rudin_PrinciplesMathematicalAnalysis_1976}.
More details on this are shown in the appendix.
Thus, we assume the gradients are well-defined and use finite-difference gradients.

Comparisons of the BWB geometries before and after each optimization problem are shown in Figure~\ref{fig:bwb_geometries}.

\begin{figure}
    \centering
    \begin{subfigure}[t]{0.3\textwidth}
        \centering
        \includegraphics[width=\textwidth]{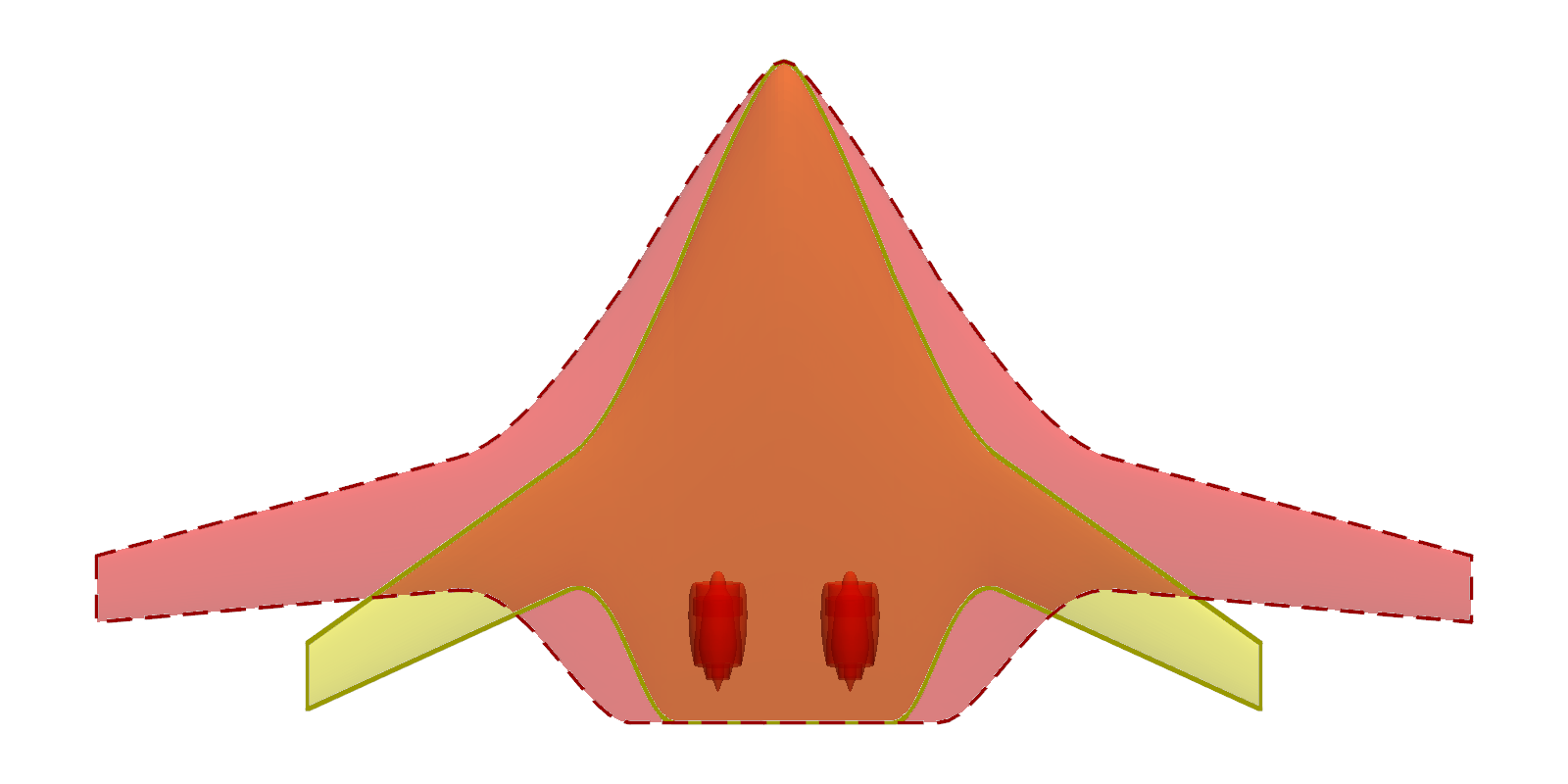}
        \caption{BWB design after problem~\eqref{eq:volume_max_optimization} shown in red with dashed outline.}
        \label{fig:volmax_geometry}
    \end{subfigure}
    \hfill
    \begin{subfigure}[t]{0.3\textwidth}
        \centering
        \includegraphics[width=\textwidth]{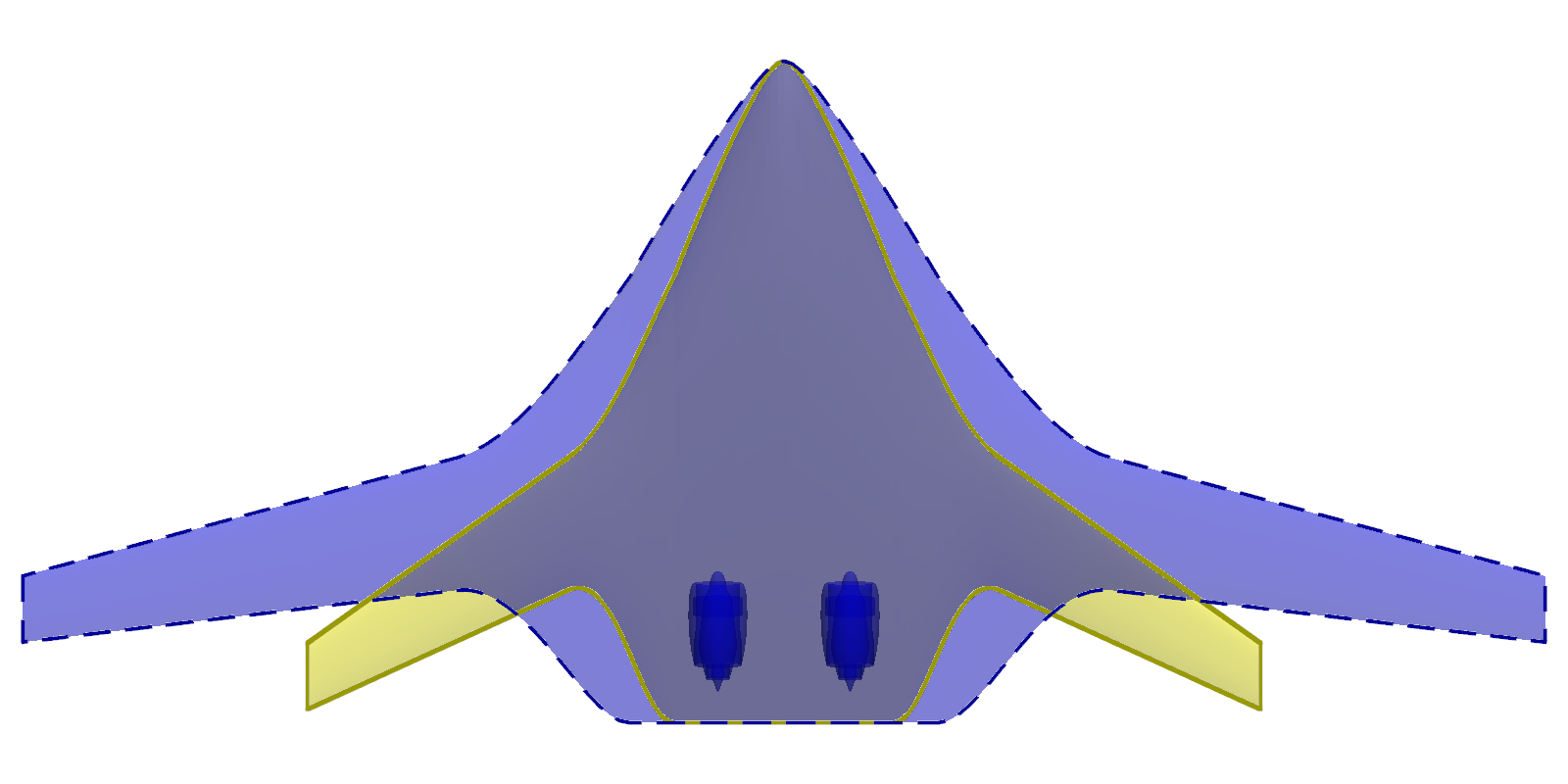}
        \caption{BWB design after problem~\eqref{eq:dirmax_optimization} shown in blue with dashed outline.}
        \label{fig:dirmax_geometry}
    \end{subfigure}
    \hfill
    \begin{subfigure}[t]{0.3\textwidth}
        \centering
        \includegraphics[width=\textwidth]{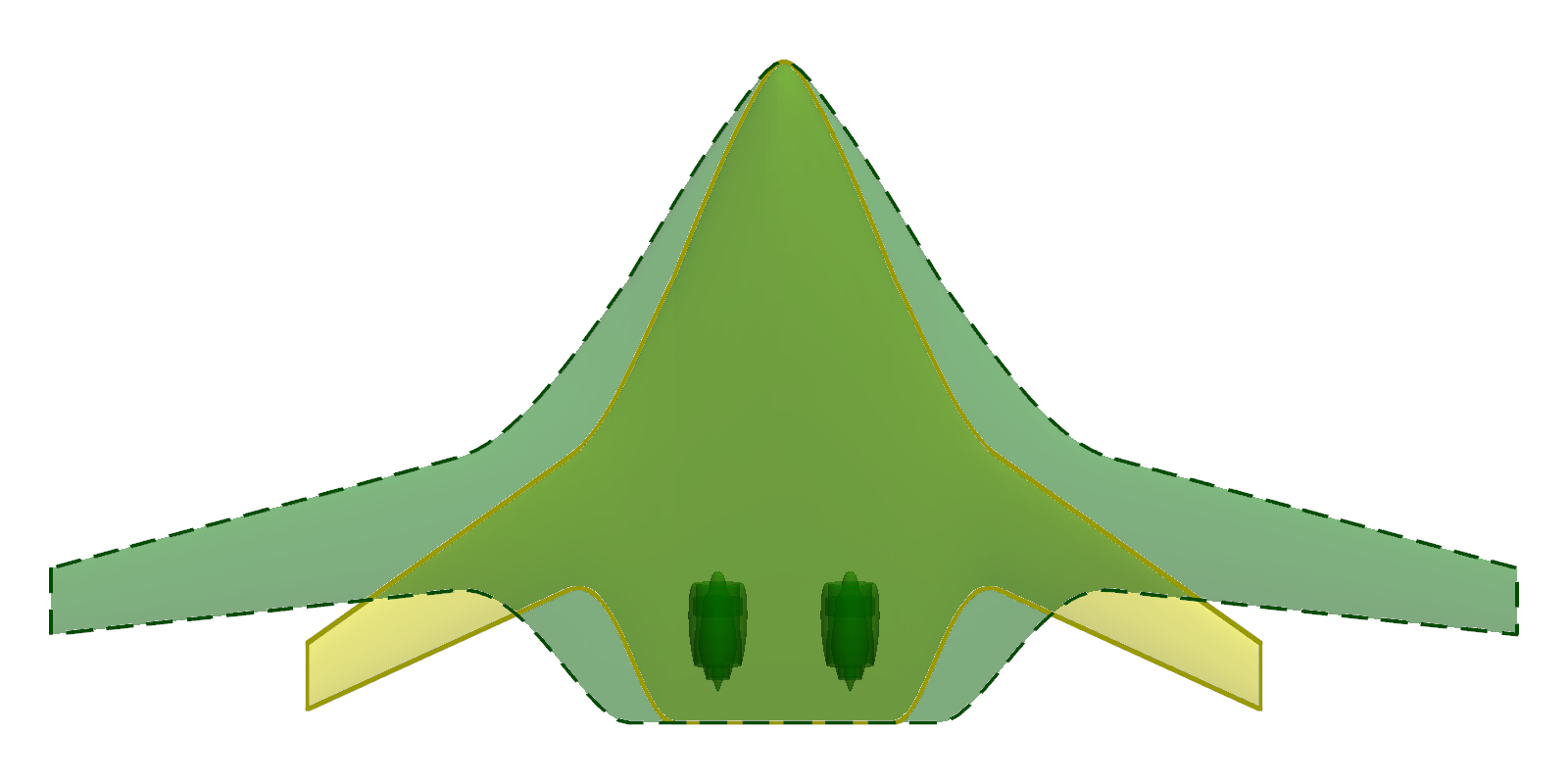}
        \caption{BWB design after problem~\eqref{eq:volmax_dircon_optimization} shown in green with dashed outline.}
        \label{fig:volmax_dircon_geometry}
    \end{subfigure}
   \caption{Comparisons of optimized BWB designs with the initial design, shown in yellow with solid outline.}
   \label{fig:bwb_geometries}
   \end{figure}   

\subsubsection{Solution of Volume Maximization}\label{subsubsec:solution_volmax}
Solving the optimization problem~\eqref{eq:volume_max_optimization}, we find optimal values of 
\begin{equation*}
    \Ra{[c^*,\ w^*,\Lambda^*,\ell^*,\Gamma^*]=[7 \textrm{ m},\ 16.6\textrm{ m},\ 15\textrm{ deg},\ 8\textrm{ m},\ 10\textrm{ deg}]^\top}.
\end{equation*}
Compared with the initial design of
\begin{equation*}
    \Ra{[c_0,\ w_0,\ \Lambda_0,\ \ell_0,\ \Gamma_0]^\top = [5 \textrm{ m},\ 12\textrm{ m},\ 35\textrm{ deg},\ 4.9\textrm{ m},\ 10\textrm{ deg}]^\top},
\end{equation*}
this is an increase in the center half-span, decrease in wing half-span, decrease in wing sweep angle, increase in transition half-span, and the same wing dihedral.
The volume of the reachable set associated with this design change increases \Ra{1,982\% from 0.0262 to 0.546}.
The reachable set is four-dimensional, and its projections into two dimensions are shown in Figures~\ref{fig:volmax_reach_set}.
Figure~\ref{fig:volmax_set_speed_aoa} reflects increases in the BWB's ability to quickly change its angle of attack and airspeed, exceeding its capabilities from before optimization.
As seen in Figure~\ref{fig:volmax_set_pitch_rate_angle}, even though the volume of the reachable set for the optimized system is larger than for the initial design of the BWB, its projection onto pitch rate and pitch angle appears more lopsided after optimization.


\begin{figure}
 \centering
 \begin{subfigure}[b]{0.49\textwidth}
     \centering
     \includegraphics[width=\textwidth]{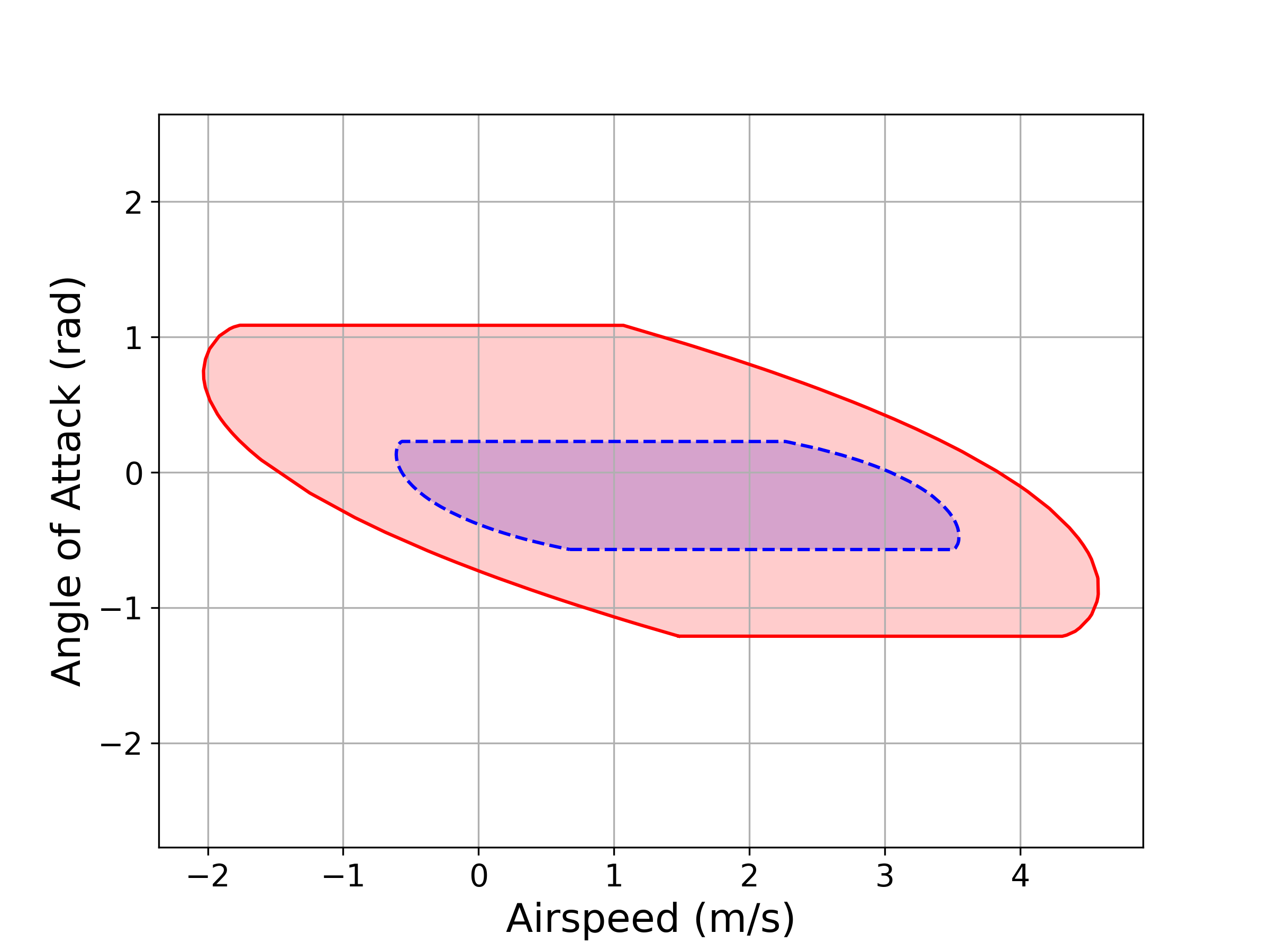}
     \caption{Projection onto the axes of airspeed and angle of attack.}
     \label{fig:volmax_set_speed_aoa}
 \end{subfigure}
 \hfill
 \begin{subfigure}[b]{0.49\textwidth}
     \centering
     \includegraphics[width=\textwidth]{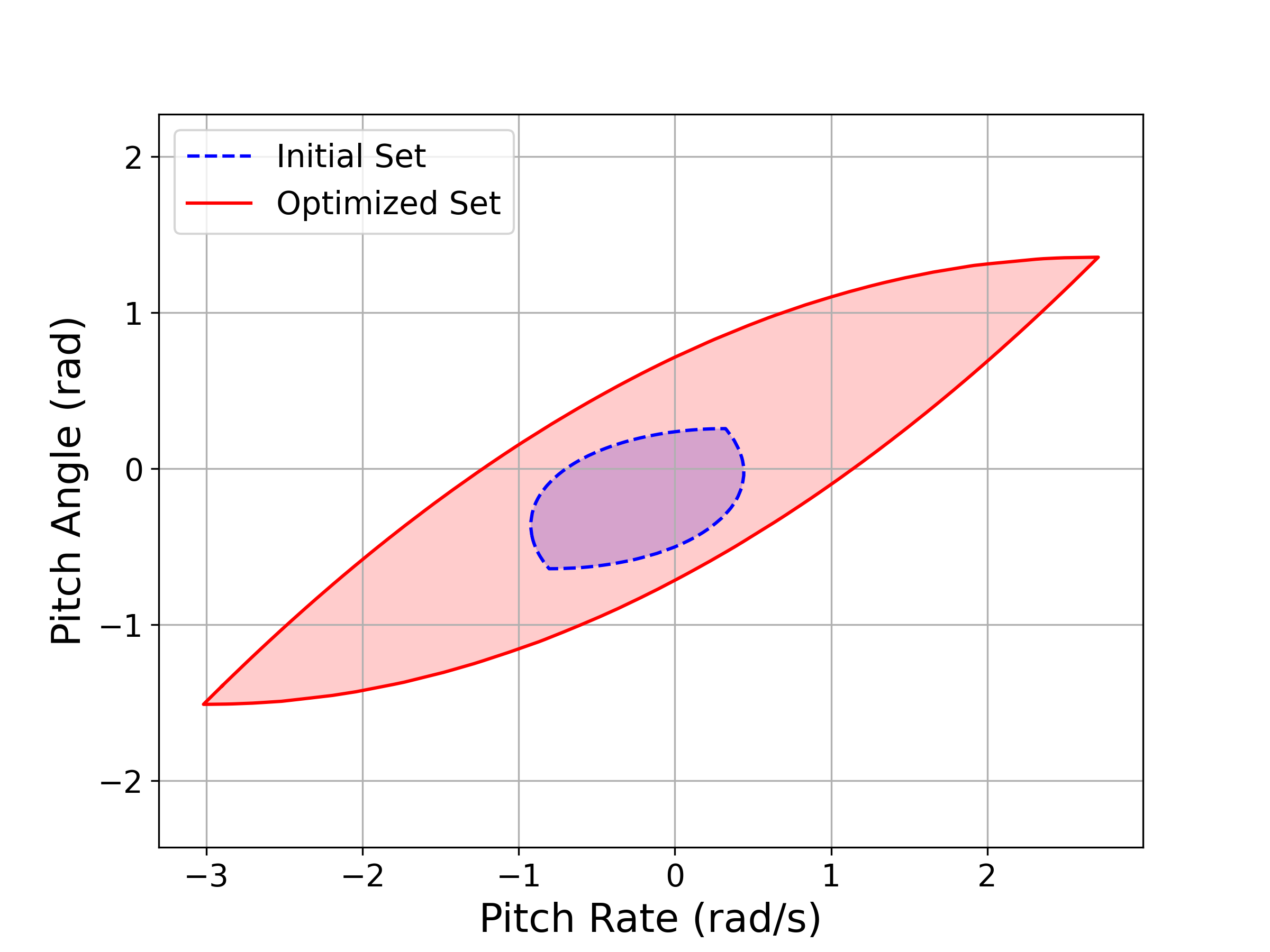}
     \caption{Projection onto the axes of pitch rate and pitch angle.}
     \label{fig:volmax_set_pitch_rate_angle}
 \end{subfigure}
\caption{Projections of four-dimensional reachable set for problem~\eqref{eq:volume_max_optimization} into two dimensions.}
\label{fig:volmax_reach_set}
\end{figure}    

\subsubsection{Solution of Direction-Maximization}\label{subsubsec:solution_dirmax}
The optimization problem~\eqref{eq:dirmax_optimization} allows us to prevent the lopsidedness in Figure~\ref{fig:volmax_set_pitch_rate_angle} if we consider its projection onto \red{$\bv = \left[  0 , 0 , \cos(115^\circ) , \sin(115^\circ) \right]^\top$}.
With this choice of $\bv$, the solution yields 
\begin{equation*}
    \Ra{[c^*,\ w^*,\Lambda^*,\ell^*,\Gamma^*]=[7 \textrm{ m},\ 20\textrm{ m},\ 15\textrm{ deg},\ 8\textrm{ m},\ 10\textrm{ deg}]^\top}.
\end{equation*} 
Compared with the solution of~\eqref{eq:volume_max_optimization}, \Ra{most design variables stay the same, but the optimal wing half-span is larger.}
Comparisons of the reachable sets and their projections onto $\bv$ for the optimized and initial designs are shown in Figure~\ref{fig:dirmax_reach_set}.
For the optimized design, the projection of the reachable set onto $\bv$ has Euclidean length \Ra{1.40, which is approximately 83.2\% larger than for the initial design, which has a projection with length 0.765}.
As seen in Figures~\ref{fig:dirmax_speed_aoa} and~\ref{fig:dirmax_pitch_rate_angle}, the focus on the reachable set projection improves the aircraft's maneuverability with respect to pitching motion, while maintaining similar maneuverability in airspeed and angle of attack.


\begin{figure}
 \centering
 \begin{subfigure}[b]{0.49\textwidth}
     \centering
     \includegraphics[width=\textwidth]{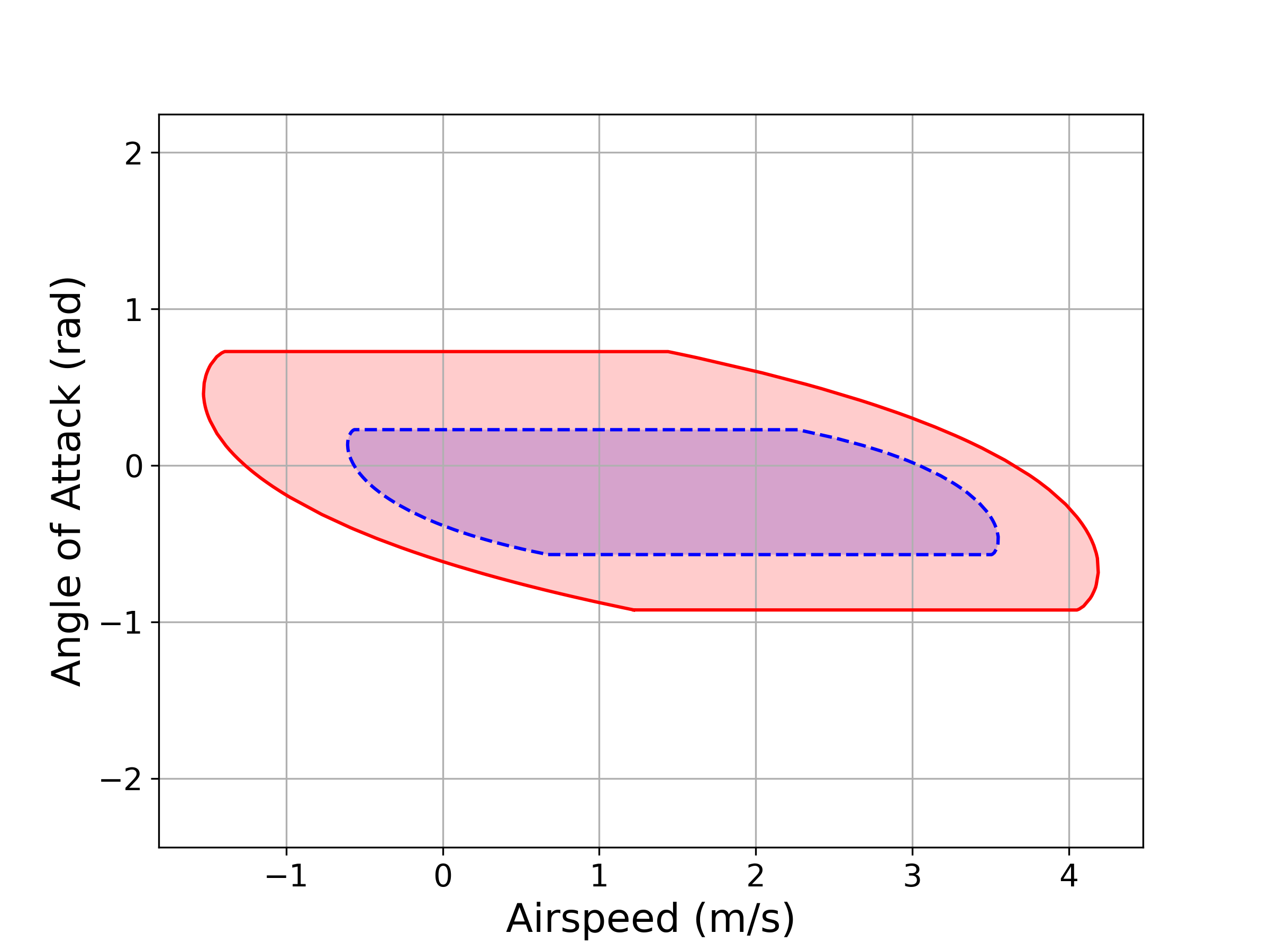}
     \caption{Projection onto the axes of airspeed and angle of attack.}
     \label{fig:dirmax_speed_aoa}
 \end{subfigure}
 \hfill
 \begin{subfigure}[b]{0.49\textwidth}
     \centering
     \includegraphics[width=\textwidth]{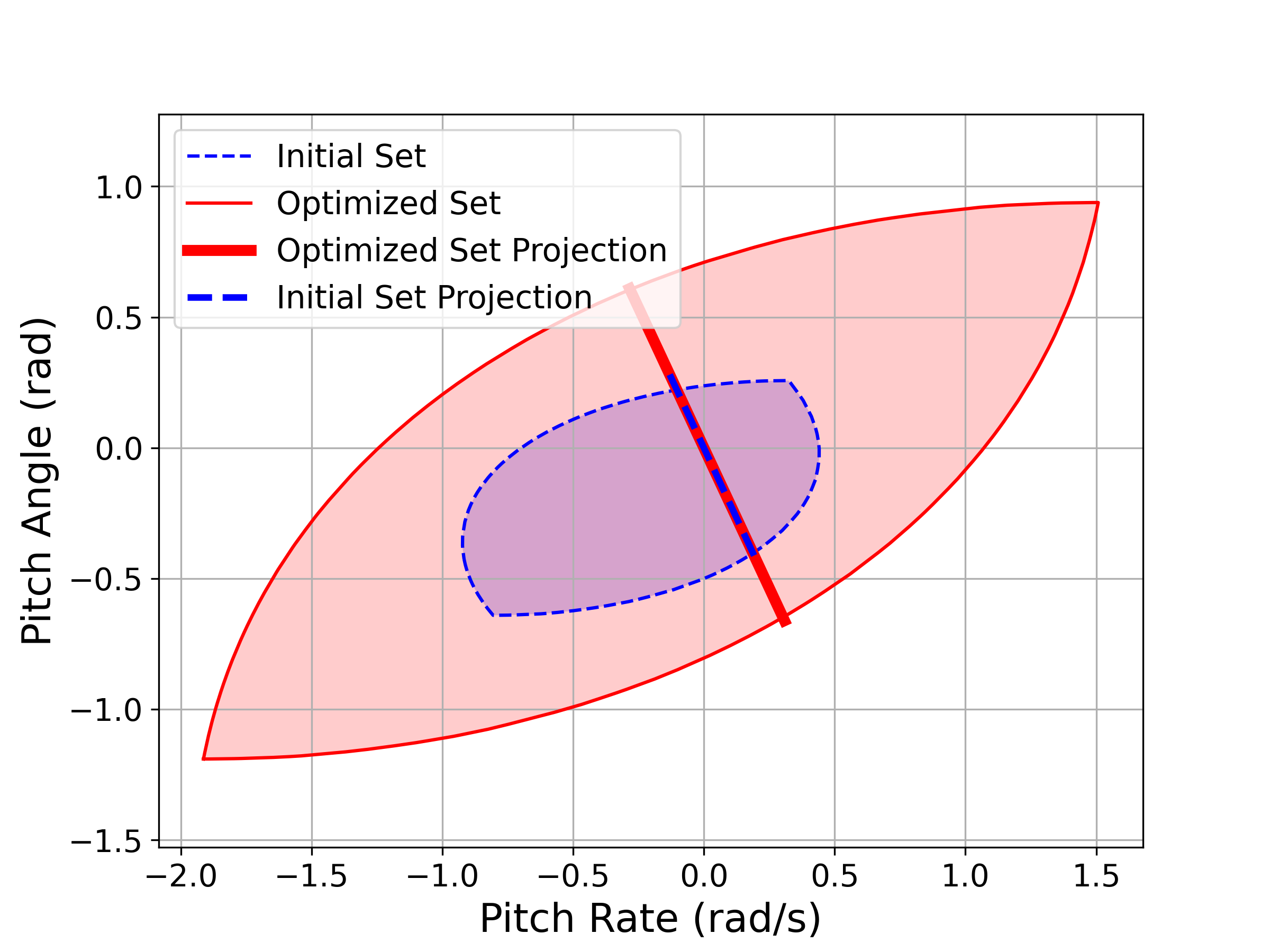}
     \caption{Projection onto the axes of pitch rate and pitch angle.}
     \label{fig:dirmax_pitch_rate_angle}
 \end{subfigure}
\caption{Projections of reachable set for problem~\eqref{eq:dirmax_optimization} into two dimensions.}
\label{fig:dirmax_reach_set}
\end{figure}

\subsubsection{Solution of Volume-Maximization with Directional Constraint}\label{subsubsec:solution_volmax_dircon}
The optimization problem~\eqref{eq:volmax_dircon_optimization} blends the results of the previous two problems.
By maximizing the volume while imposing that the projection along $\bv$ increases by \red{80\% (using $\kappa = 0.8$)}, we find 
\begin{equation*}
    \Ra{[c^*,\ w^*,\Lambda^*,\ell^*,\Gamma^*]=[7 \textrm{ m},\ 18.7\textrm{ m},\ 15\textrm{ deg},\ 8\textrm{ m},\ 10\textrm{ deg}]^\top}.
\end{equation*}
Similarly to the previous optimization problems~\eqref{eq:volume_max_optimization} and~\eqref{eq:dirmax_optimization}, \Ra{the optimal values of all design variables remain the same except for the optimal wing half-span, $w$, which lies between the optimal values for the other two problems.}
These design variables lead to a \Ra{1,594\% increase in the reachable set volume from 0.0262 to 0.445}.
The projection of the optimized reachable set onto $\bv$ is \Ra{1.38, which is 80\%} larger than for the initial design and satisfies the directional constraint.
Comparisons between the reachable sets and their projections of the initial and optimized design for~\eqref{eq:volmax_dircon_optimization} are shown in Figure~\ref{fig:volmax_dircon_reach_set}.
It is clear from Figure~\ref{fig:volmax_dircon_pitch_rate_angle} that the additional constraint on the set projection prevents the optimized reachable set from becoming overly lopsided, while still increasing the volume of the reachable set.

\Ra{We note that four of the five design variables have the same optimal value between the three optimization problems.
The effect of these design variable choices is to enhance the BWB's maneuverability by improving its capability for force and moment generation to achieve greater accelerations.
It is intuitive that the center half-span and transition half-span are maximized because increasing these effectively increases the aspect ratio of the BWB, which increases its lift-curve slope~\cite{anderson_FundamentalsAerodynamics_1991}, and increases its capacity to generate larger pitching moments.
The wing sweep angle is at its lower bound, which is reasonable because a lower sweep angle decreases span-wise flow and allows the aircraft to generate greater lift forces.
It is also expected that the wing dihedral is at its lower bound because a larger dihedral creates components of lift that counterract each other and reduce the aircraft's maximum lift force.
Since we do not consider lateral dynamics in this study, the lateral stability offered by an increased wing dihedral has no effect.
Finally, the wing half-span offers direct control over the wing aspect ratio and surface area.
Consequently, it is intuitive that in the~\eqref{eq:dirmax_optimization} case, this is maximized to improve the aircraft's ability to generate large pitching moments and improve reachability in pitch axis and rate.
However, in the~\eqref{eq:volume_max_optimization} case, the fact that wing half-span takes a lower value may indicate that the greater drag incurred from increasing the wingspan may outweigh the improvements in pitching capabilities when measured in terms of the reachable set volume.}
\red{The static margin constraint can be modified to enforce a greater amount of stability and can lead to results where the center half-span and wing sweep are no longer at their bounds.
However, enforcing a greater static margin limits the aircraft's improvements in maneuverability, which we discuss in the following section.}

\begin{figure}
 \centering
 \begin{subfigure}[b]{0.49\textwidth}
     \centering
     \includegraphics[width=\textwidth]{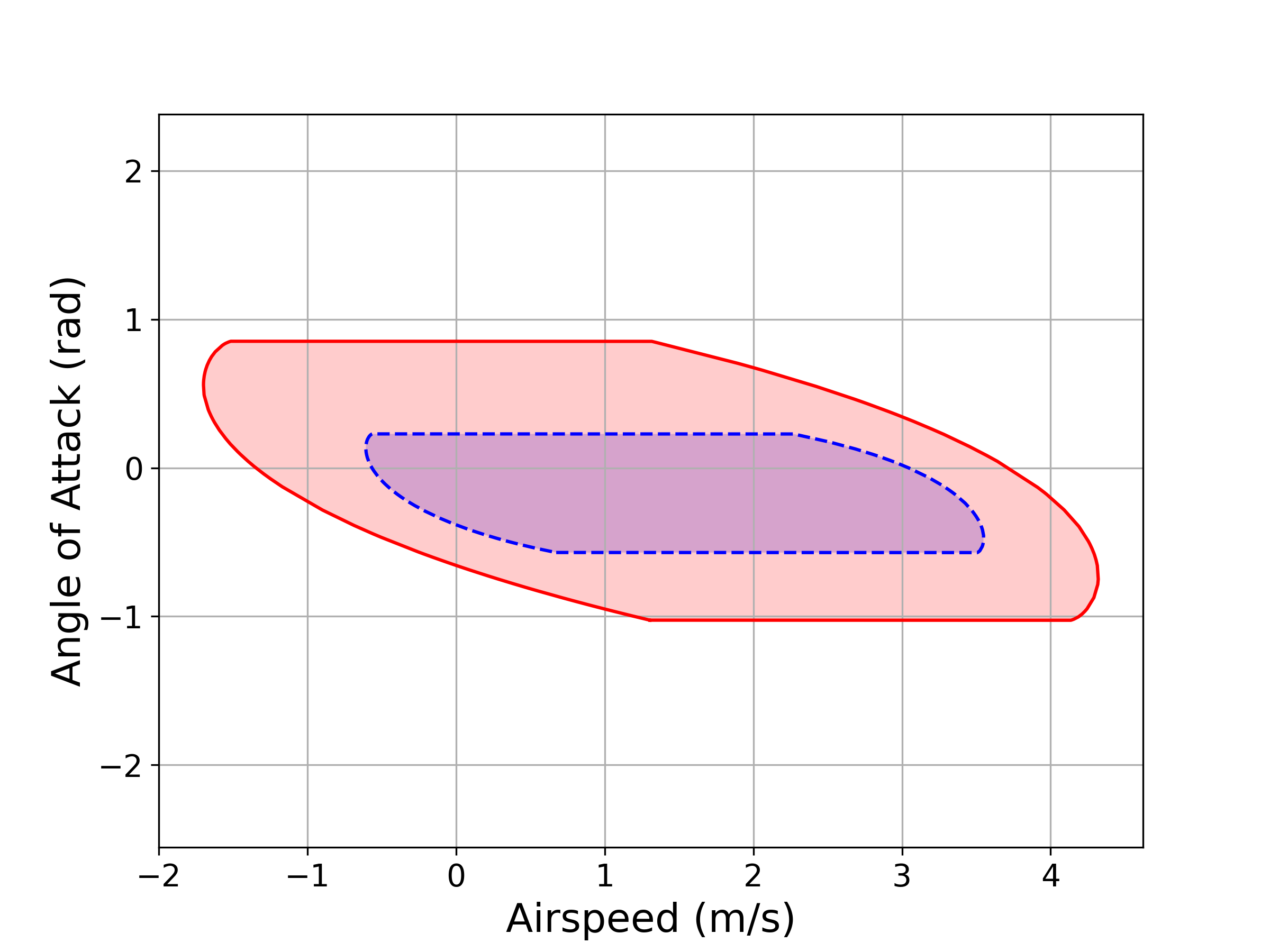}
     \caption{Projection onto the axes of airspeed and angle of attack.}
     \label{fig:volmax_dircon_speed_aoa}
 \end{subfigure}
 \hfill
 \begin{subfigure}[b]{0.49\textwidth}
     \centering
     \includegraphics[width=\textwidth]{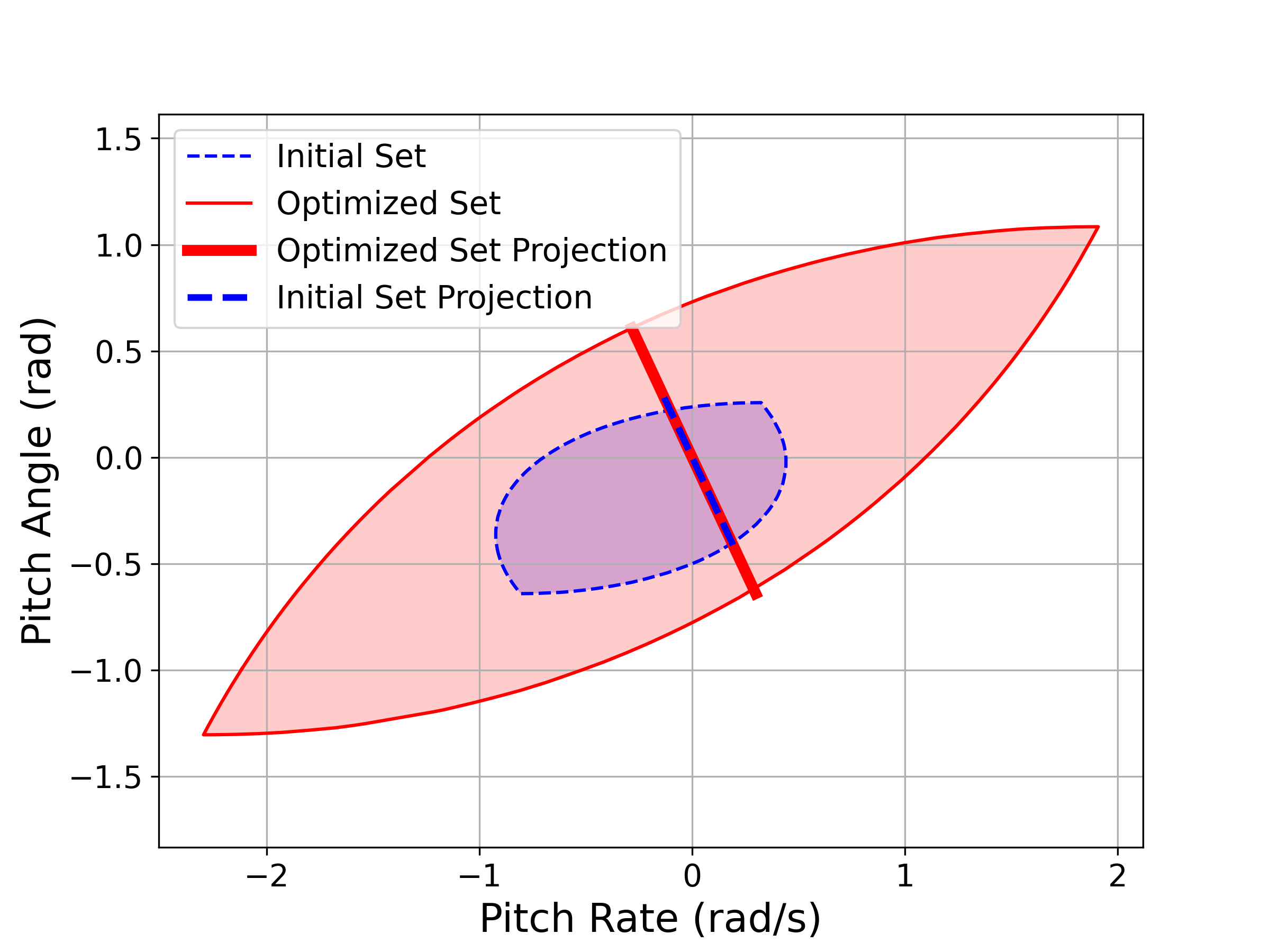}
     \caption{Projection onto the axes of pitch rate and pitch angle.}
     \label{fig:volmax_dircon_pitch_rate_angle}
 \end{subfigure}
\caption{Projections of reachable set for problem~\eqref{eq:volmax_dircon_optimization} into two dimensions.}
\label{fig:volmax_dircon_reach_set}
\end{figure}

\section{Improvements in Controlled Performance Through Reachability-Based Design Optimization}\label{sec:control_analysis}

We analyze the controlled capabilities of the BWB aircraft by comparing its performances in control tasks using the initial and optimized aircraft designs.
\red{Taking advantage of the controller-agnostic nature of reachability analysis, we consider an optimal linear-quadratic (LQ) reference tracking controller for the linearized dynamics in Section~\ref{subsec:lq_tracking} and a linear-quadratic-integral (LQI) controller for nonlinear reference tracking in Section~\ref{subsec:nonlinear_tracking}.}
\Rb{In Section~\ref{subsec:flight_perf_metr}, we compare the flight performance metrics of the different designs.}
\Rb{Lastly, in Section~\ref{subsec:gramian_comp}, we compare our reachability-based approach to design optimization with a Gramian-based approach~\cite{gupta_ControllabilityGramianControl_2020}.}
In all the following simulations, we apply bounds on the inputs to capture the effects of actuator saturation.
To compare the controlled performance between designs, we use the $\Lp[2]$-norm, defined for $\bsigma:[t_0,T] \to \R^{n_\sigma}$ as:
\begin{equation}
    \Lpnorm[2]{\bsigma(\cdot)} = \left( \int\limits_{t_0}^{T} \sum\limits_{i=1}^{n_\sigma} \sigma_i(t)^2 \rd t \right)^{{1}/{2}}.
\end{equation}

\subsection{Optimal Linear-Quadratic Tracking}\label{subsec:lq_tracking}

We apply the optimal linear-quadratic (LQ) reference tracking controller, as outlined in~\cite{anderson2007optimal}, to the linearized model~\eqref{eq:lavret_wise_linear_long_model}.
Given a starting point $\bx_0$ at time $t_0$, this controller minimizes the quadratic cost:
\begin{align}\label{eq:quadratic_cost}
    \bu^*(\cdot\,; \bxi[0],t_0,T) = \argmin_{\bu(\cdot):[t_0,T]\to\calU} \int\limits^T_{t_0} & \left[ \bu(\tau)^\top \bolds{R} \bu(\tau)\right. \\ \nonumber
                                & + \left. \left( \bx(\tau;\bxi[0],t_0,\bu(\cdot)) - \bxi[\textrm{ref}](\tau) \right)^\top \bolds{Q} \left( \bx(\tau; \bxi[0],t_0,\bu(\cdot)) - \bxi[\textrm{ref}](\tau) \right) \right] \rd \tau ,
\end{align}
where $\bxi[\textrm{ref}](\cdot)$ is the reference trajectory, $\bolds{Q}$ is symmetric nonnegative definite, and $\bolds{R}$ is symmetric positive definite.
The reference is assumed to be given in its entirety a priori.

We consider two separate cases of a constant velocity reference at four m/s and constant pitch reference at 0.5 rad.
For the case of velocity tracking, we use $\bolds{Q} = \diag([1000, \ 0,\ 0 ,\ 0]) $ and $\bolds{R} = \diag([1000,\ 1000])$, where $\diag(\bolds{\beta})$ refers to the matrix with elements of $\bolds{\beta}$ on its diagonal.
For the pitch tracking, we use $\bolds{Q} = \diag([0, \ 0,\ 0 ,\ 1000]) $ and $\bolds{R} = \diag([100,\ 100])$
The system trajectories and control signals for the initial and three optimized designs are shown in \Rb{Figures~\ref{fig:lqr_pitch_all} and~\ref{fig:lqr_vel_all}}.
The $\Lp[2]$ norms of the tracking errors and control costs of each design in both the velocity and pitch tracking problems are shown in Table~\ref{tab:tracking_error_lq}.

\begin{figure}[h]
    \centering
    \includegraphics[width=0.7\linewidth]{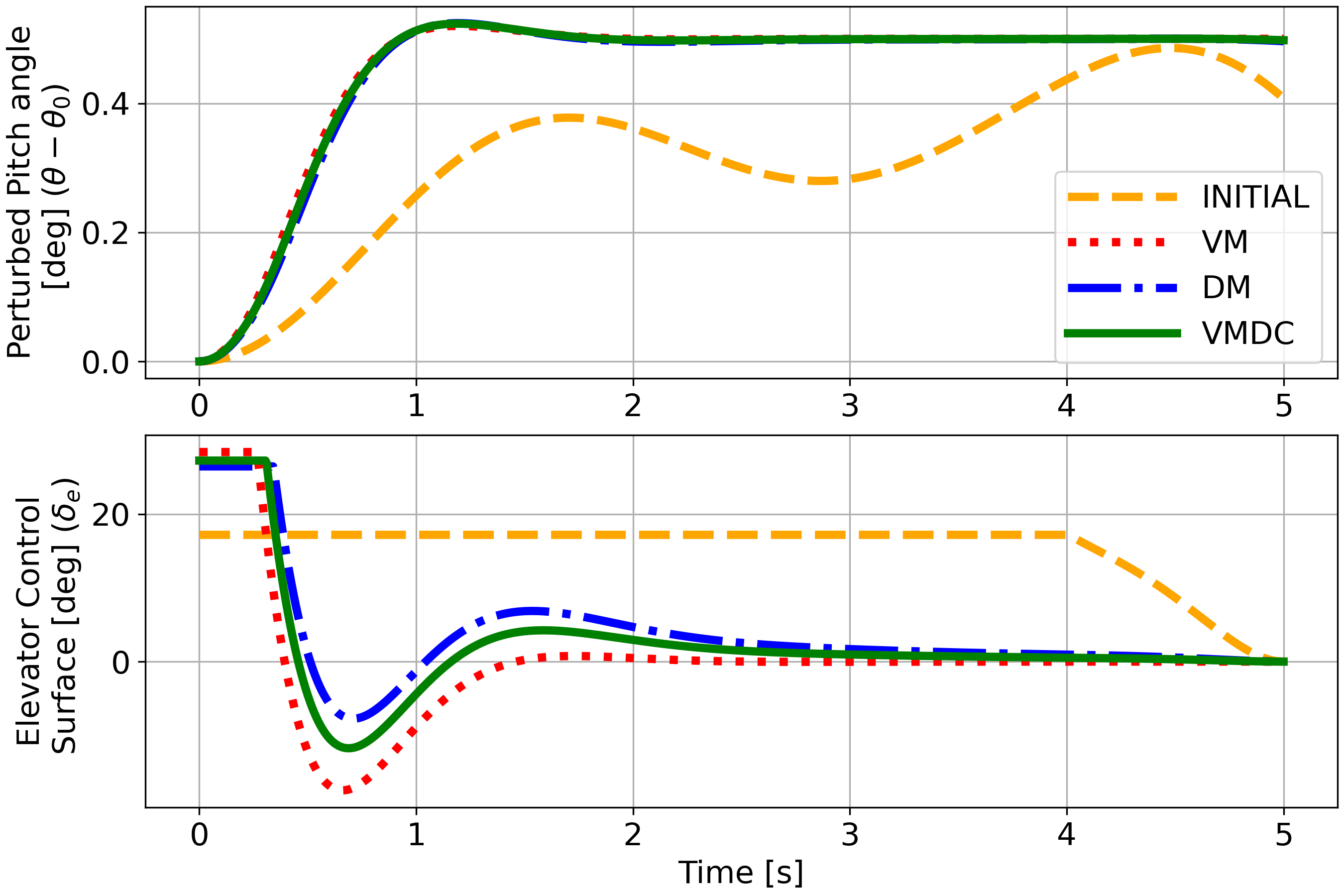}
    \caption{\Rb{Comparisons of the BWB linear-longitudinal dynamics for optimal LQ reference tracking of the pitch angle reference between the initial design, and the designs from~\eqref{eq:volume_max_optimization},~\eqref{eq:dirmax_optimization}, and~\eqref{eq:volmax_dircon_optimization}.
    Thrust control is omitted from the plot because the behavior of the thrust did not change significantly between models.}}
    \label{fig:lqr_pitch_all}
\end{figure}

\begin{figure}[h]
    \centering
    \includegraphics[width=0.75\linewidth]{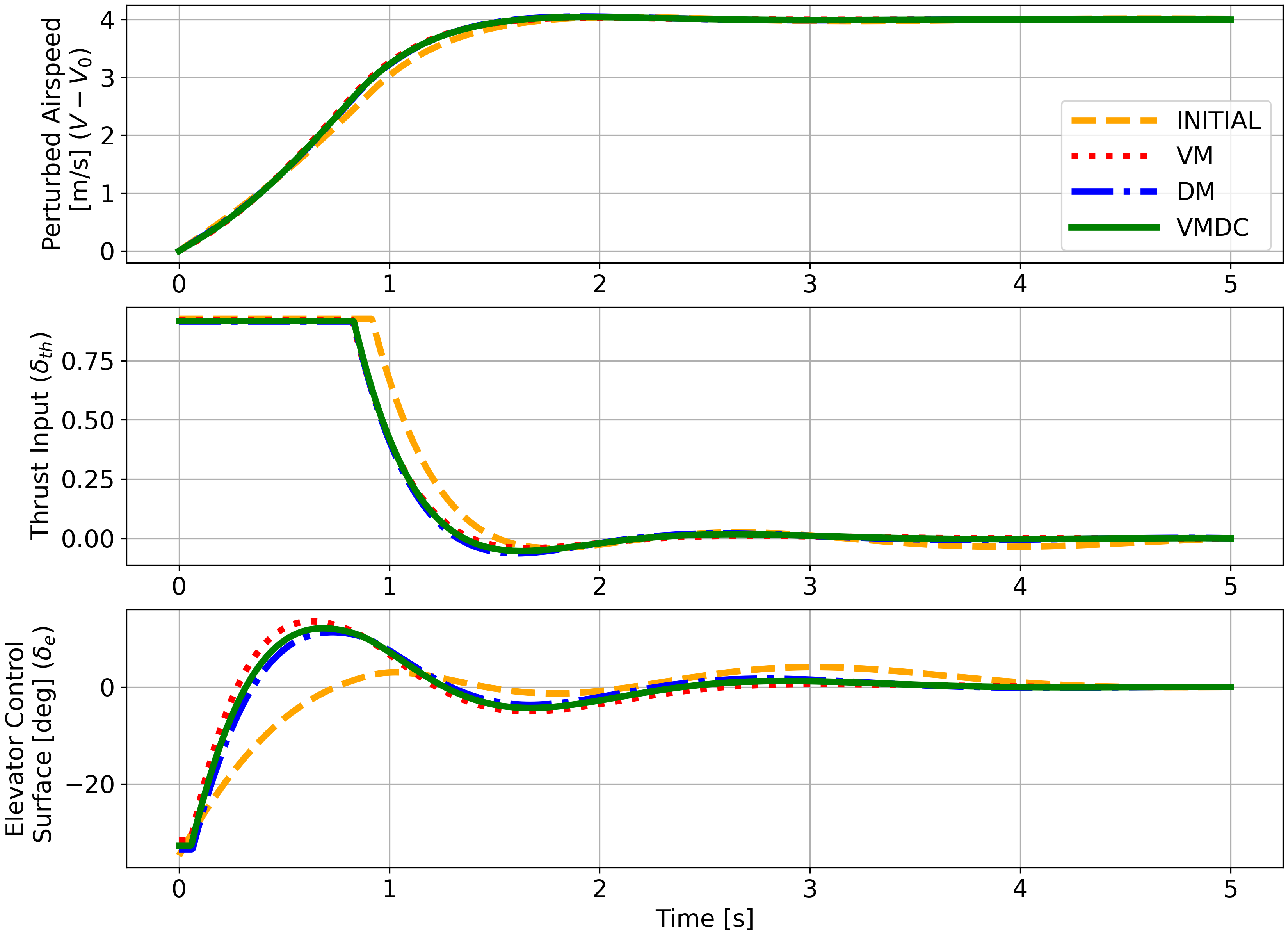}
    \caption{\Rb{Comparisons of the BWB linear-longitudinal dynamics for optimal LQ reference tracking of the velocity reference between the initial design, and the designs from~\eqref{eq:volume_max_optimization},~\eqref{eq:dirmax_optimization}, and~\eqref{eq:volmax_dircon_optimization}.}}
    \label{fig:lqr_vel_all}
\end{figure}

\begin{table}[ht]
\centering
\begin{tabular}{C{2.5cm}|C{2.5cm}|C{2.5cm}|C{2.5cm}|C{2.5cm}}
Design & Velocity Tracking $\Lp[2]$ Error (\%~Improvement) & Velocity Tracking Elevator $\Lp[2]$ Cost (\%~Improvement) & Pitch Tracking $\Lp[2]$ Error (\%~Improvement) & Pitch Tracking Elevator $\Lp[2]$ Cost (\%~Improvement) \\ \hline

Initial Model                                                   & \Ra{2.76         } & \Ra{0.267          } & \Ra{0.505         } & \Ra{0.625         } \\ \hline
Vol. Max. Opt.~\eqref{eq:volume_max_optimization}               & \Ra{2.71 (1.6\%)} & \Ra{0.259 (3.2\%)}  & \Ra{0.293 (42.0\%)} & \Ra{0.335 (46.3\%)} \\ \hline
Dir. Max. Opt.~\eqref{eq:dirmax_optimization}                   & \Ra{2.72 (1.3\%)} & \Ra{0.264 (1.4\%) } & \Ra{0.304 (39.9\%)} & \Ra{0.315 (49.5\%)} \\ \hline
Vol. Max. Dir. Con. Opt.~\eqref{eq:volmax_dircon_optimization}  & \Ra{2.72 (1.3\%)} & \Ra{0.258 (3.4\%)}  & \Ra{0.299 (40.8\%)} & \Ra{0.311 (50.2\%)} \\
\end{tabular}
\caption{Tracking errors and \red{elevator} control costs in optimal LQ reference tracking of velocity and pitch commands. Percent improvement shows decrease in cost or error with respect to initial design (positive is reduction, negative is increase).} 
\label{tab:tracking_error_lq}
\end{table}

The volume-maximized~\eqref{eq:volume_max_optimization} \red{design has} lower tracking errors in \Ra{both the velocity and pitch cases than the initial design. 
Furthermore, it achieves these improvements in reference tracking with less control effort, using less actuation from the elevator than the initial design}.
This result is intuitive because this design is optimized for the reachable set volume, \Ra{which we interpret as a general characterization of maneuverability}.
Although the reachable set is lopsided in the pitch rate and pitch angle axes, as seen in Figure~\ref{fig:volmax_set_pitch_rate_angle}, this design features the \Ra{largest decreases in tracking error when compared with the other optimized designs.
Since the reachable sets characterize the set of states that can be achieved in finite time, it makes sense that the larger set of reachable pitch angles leads to improved ability to track changes in that state.}

The direction-maximized~\eqref{eq:dirmax_optimization} \red{design has} \Ra{less tracking error than the initial design in both cases, reducing velocity tracking error by 1.3\%, and pitch tracking error by 39.9\%.
As with the~\eqref{eq:volume_max_optimization} design, these tracking improvements are accompanied by reduced elevator control effort when compared with the initial design.
When compared with~\eqref{eq:volume_max_optimization}, this design has slightly larger tracking error in both cases, but uses less elevator control effort in the pitch tracking case, which may signal that the choice of projection direction $\bv$ can be tuned to balance these tradeoffs.}
The vector $\bv$ was chosen to prevent the lopsidedness of the reachable set in Figure~\ref{fig:volmax_reach_set}, but is not necessarily the optimal choice for improving maneuverability in either velocity or pitch tracking.

Lastly, the volume-maximized, directionally constrained~\eqref{eq:volmax_dircon_optimization} \red{design appears} to split the difference between the~\eqref{eq:volume_max_optimization} and~\eqref{eq:dirmax_optimization} designs, \Ra{with tracking errors slightly larger than for~\eqref{eq:volume_max_optimization}, but less than~\eqref{eq:dirmax_optimization}.
However, this design features the greatest reduction in elevator control cost across all optimized designs, which may signify, with careful choice of $\bv$ and $\kappa$, that combining the metrics of reachable set volume and projection can be more effective than either approach on its own.}

\subsection{Nonlinear Reference Tracking}\label{subsec:nonlinear_tracking}
Although the design optimization problems~\eqref{eq:volume_max_optimization},~\eqref{eq:dirmax_optimization}, and~\eqref{eq:volmax_dircon_optimization} leverage reachable sets based on linearized dynamics, the linear analysis still provides performance benefits when evaluating the \textit{nonlinear longitudinal dynamics}~\eqref{eq:nonlinear_longitudinal_model}.
To compare the controlled performance of the nonlinear dynamics before and after optimizing the BWB, we consider a nonlinear control task that first completes a quasi-steady climb before transitioning to level flight.
The first maneuver is to climb with a flight path angle of $\gamma = 10^\circ$ at $V_0 = 190$ m/s for \red{ten seconds}.
The second maneuver is to fly level at $V_0 = 210$ m/s for \red{ten seconds}.
We design two controllers, using LQR with an integrator on the pitch angle, around the trim points for both of these maneuvers, and we set an initial condition of \red{$V_0 = 185$ m/s and $\gamma = 5^\circ$}.
The cost matrices are $\bolds{Q} = \diag([1, \ 100,\ 1 ,\ 100,\ \red{50}]) $ and $\bolds{R} = \diag([0.1,\ 1000])$.
We use the sampled aerodynamic data from solving the design optimization problem, as described in Section~\ref{subsec:aero_data_collection}, to capture the drag, lift, and pitching moment the aircraft experiences at each point in time.

Trajectories comparing each optimized model with the initial model are presented in Figure~\ref{fig:3_models_nonlinear_track}.
\Ra{Since each model has a different trim point, these figures plot the differences between the states of the nonlinear systems and their desired reference points.
The spikes in each plot of Figure~\ref{fig:3_models_nonlinear_track} at the ten second marks signify the switch to the second maneuver and new reference point.}
\red{The thrust control for each design is not shown in Figure~\ref{fig:3_models_nonlinear_track} because it remained similar between each design.}
Similarly to Table~\ref{tab:tracking_error_lq} for the linear simulations, Table~\ref{tab:tracking_error_nonlinear} shows the $\Lp[2]$-norms of the aircraft deviations from trim, $\bx - \bx_0$, and $\Lp[2]$-norms of the \red{elevator control} costs.
\red{Transient performance metrics, including rise time, settling time, and overshoot, are shown in Table~\ref{tab:transient_3_designs}.}
\red{The errors are also visualized} with a bar graph in Figure~\ref{fig:error_nonlinear}.




\begin{figure}[ht]
    \centering
    \includegraphics[width=\linewidth]{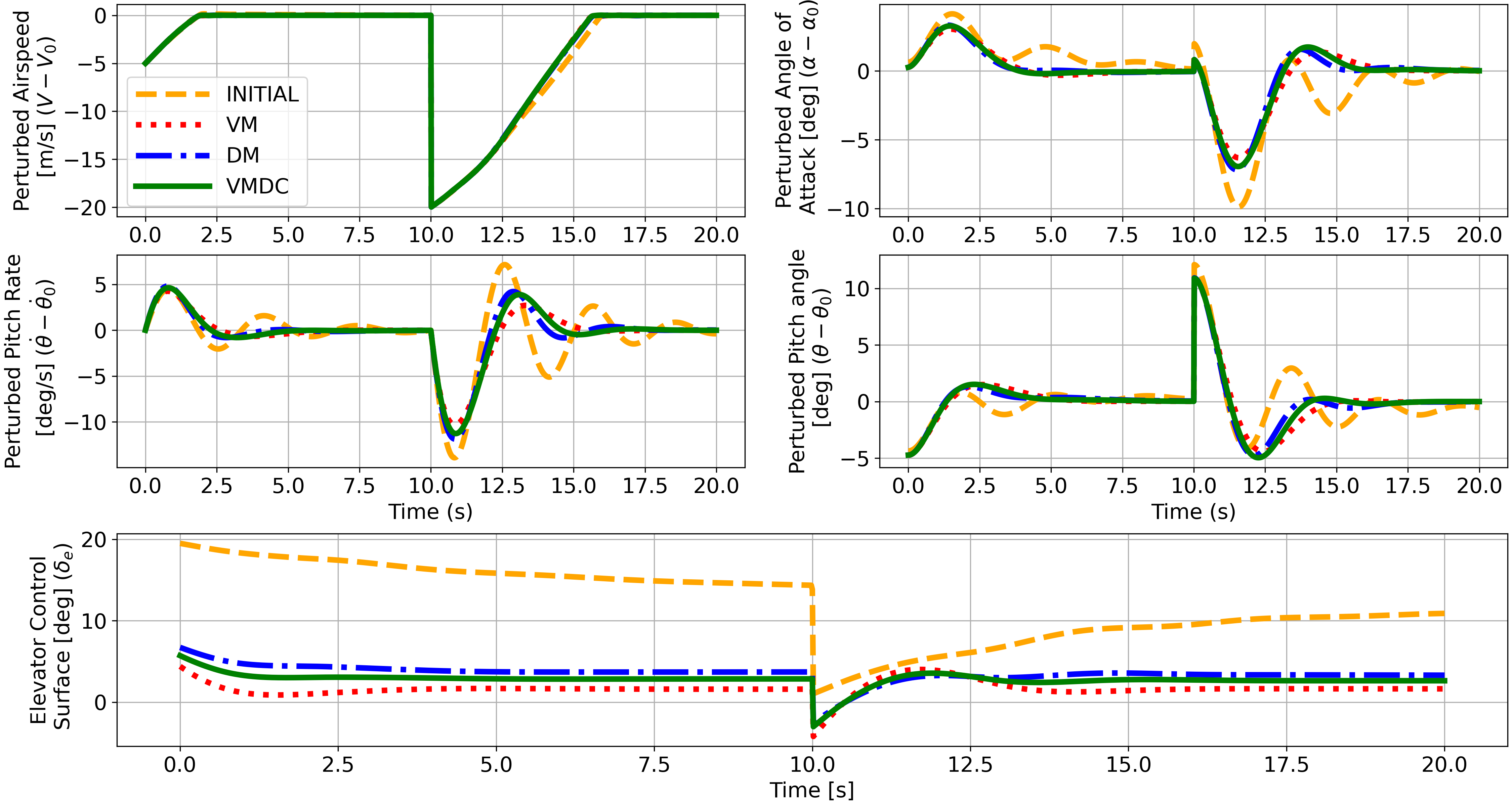}
    \caption{\Rb{Deviation of nonlinear dynamics from trimmed reference points between initial design and designs using~\eqref{eq:volume_max_optimization}, \eqref{eq:dirmax_optimization}, and~\eqref{eq:volmax_dircon_optimization}.
    Elevator control signal is shown, while thrust control is omitted.}}
    \label{fig:3_models_nonlinear_track}
\end{figure}

\Ra{It is clear from the plots of the system responses that the optimized designs have improved tracking performance over the initial design, as the tracking errors decrease for every state.
In addition, each optimized design is able to achieve these improvements with greatly reduced elevator costs, with the smallest improvement being 71.8\%.
In each case, the improvements in stabilization of the aircraft velocity are small because the thrust input only affects the velocity and can stabilize it with minimal effect on the other states.
The most drastic improvements can be seen in the pitch rate and angle of attack signals, where the optimization problems~\eqref{eq:volume_max_optimization},\eqref{eq:dirmax_optimization}, and~\eqref{eq:volmax_dircon_optimization} achieves reductions in error of at least 16.9\% and up to 31\%.
Improvements in tracking become clearer when we consider the transient performance metrics, in Table~\ref{tab:transient_3_designs}.
While the rise time and overshoot in each of the optimized designs seem worse than for the initial design, the main improvements are seen in settling time.
All optimized designs converge to the reference trajectory with settling times at least 34.2\% faster than the initial design, and up to 57.1\% faster.
Similarly to the linear simulations, the~\eqref{eq:volmax_dircon_optimization} design exhibits tracking errors and an elevator control cost that fall in between the~\eqref{eq:volume_max_optimization} and~\eqref{eq:dirmax_optimization} designs.
However, the nonlinear simulations showcase that the~\eqref{eq:volmax_dircon_optimization} can achieve the fastest settling time, signifying again that combining the volume and projection-based reachability analysis can be more effective than either method applied separately.}

\begin{table}[ht]
    \centering
    \begin{tabular}{C{1.75cm}|C{2.25cm}|C{2.25cm}|C{2.25cm}|C{2.25cm}|C{2.25cm}}
    Design &  Velocity $\Lp[2]$ Error (\%~Improvement)& Angle of Attack $\Lp[2]$ Error (\%~Improvement) &  Pitch Rate $\Lp[2]$ Error (\%~Improvement) & Pitch Angle $\Lp[2]$ Error (\%~Improvement) & Elevator $\Lp[2]$ Cost (\%~Improvement) \\ \hline
    Initial Design                                         &  \Ra{30.6        }   & \Ra{0.220          } & \Ra{0.286         }  & \Ra{0.189         }    & \Ra{1.01          }  \\ \hline
    Opt. Design with~\eqref{eq:volume_max_optimization}    &  \Ra{30.1 (1.6\%)}   & \Ra{0.151 (31.0\%) } & \Ra{0.215 (24.8\%)}  & \Ra{0.185 (2.1\%) }    & \Ra{0.148 (85.4\%) } \\ \hline
    Opt. Design with~\eqref{eq:dirmax_optimization}        &  \Ra{30.0 (1.9\%)}   & \Ra{0.159 (27.5\%) } & \Ra{0.237 (16.9\%)}  & \Ra{0.179 (5.5\%) }    & \Ra{0.286 (71.8\%) }  \\ \hline
    Opt. Design with~\eqref{eq:volmax_dircon_optimization} &  \Ra{30.0 (1.9\%)}   & \Ra{0.160 (27.1\%) } & \Ra{0.233 (18.5\%)}  & \Ra{0.185 (2.1\%) }    & \Ra{0.225 (77.7\%) }   \\
    \end{tabular}
    \caption{Tracking errors and control costs for control of nonlinear dynamics. Percent improvement shows decrease in cost or error with respect to initial design (positive is reduction, negative
    is increase)}
    \label{tab:tracking_error_nonlinear}
\end{table}

\begin{table}[ht]
    \centering
    \begin{tabular}{C{2.5cm}|C{2.5cm}|C{2.5cm}|C{2.5cm}}
    Design &  Rise Time (s) (\%~Improvement) & Settling Time (s) (\%~Improvement) & Overshoot (\%) (\%~Improvement) \\ \hline
    Initial Design                                         &  \red{1.14            }  & \red{8.80            } & \red{34.2           } \\ \hline
    Opt. Design with~\eqref{eq:volume_max_optimization}    &  \red{1.30 ($-14.7\%$)}  & \red{4.36 ($50.5\%$)} & \red{40.5 ($-18.4\%$)  } \\ \hline
    Opt. Design with~\eqref{eq:dirmax_optimization}        &  \red{1.17 ($-2.9\%$)   }  & \red{5.79 (34.2\%)   } & \red{43.6 ($-27.3\%$)  } \\ \hline
    Opt. Design with~\eqref{eq:volmax_dircon_optimization} &  \red{1.22 ($-7.4\%$)}  & \red{3.77 ($57.1\%$) } & \red{45.6 ($-33.2\%$)} \\
    \end{tabular}
    \caption{\red{Pitch angle rise time, settling time, and overshoot for nonlinear reference tracking between designs using the proposed reachability metrics. These metrics are computed for the transition between climb and level-flight, at ten seconds. The percent improvements with respect to the initial design are shown in parentheses (positive is reduction, negative is increase).}}
    \label{tab:transient_3_designs}
\end{table}

\begin{figure}[ht]
    \centering
    \includegraphics[width=0.65\linewidth]{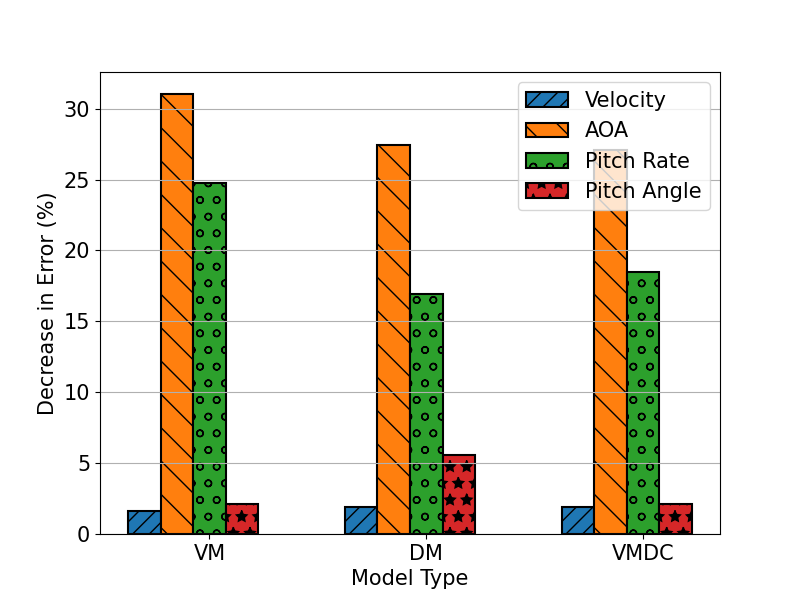}
    \caption{Decrease in $\Lp[2]$ tracking error between initial design and optimized designs~\eqref{eq:volume_max_optimization},\eqref{eq:dirmax_optimization}, and~\eqref{eq:volmax_dircon_optimization} in nonlinear control task.}
    \label{fig:error_nonlinear}
\end{figure}

\subsection{\Rb{Flight Performance Metrics}}\label{subsec:flight_perf_metr}


\Rb{To analyze flight performance metrics, we compute the aerodynamic efficiency, computed as lift divided by drag, and power required for the BWB in level-flight at 200 m/s.
For both metrics, we find that the initial design outperforms all the optimized designs, with an efficiency of 31.1 and required power of 9.4 MW.
In comparison, the~\eqref{eq:dirmax_optimization} design had the worst metrics, with an efficiency of 27.1 and required power of 10.9 MW.
Among the optimized designs, the~\eqref{eq:volume_max_optimization} had the best metrics, with an efficiency of 28.5 and required power of 10.3 MW.
The metrics for the~\eqref{eq:volmax_dircon_optimization} design fall in between those of~\eqref{eq:volume_max_optimization} and~\eqref{eq:dirmax_optimization}.}

\Rb{While the reachability-optimized designs successfully reduce control effort and improve reference tracking performance, their aerodynamic efficiency and power required are worse than the initial design.
We interpret this as an indication that, while the proposed metrics emphasize the aircraft's ability to quickly change states like its angle of attack or airspeed, they do not emphasize reducing the amount of drag the aircraft faces in a given operating condition.
Although the optimized designs required more thrust power to maintain level-flight at 200 m/s, their ability to more rapidly change in pitch angle and angle of attack allowed for greater changes in airspeed as well, as shown in the reachable set projections (Figures~\ref{fig:volmax_set_speed_aoa},~\ref{fig:dirmax_speed_aoa}, and~\ref{fig:volmax_dircon_speed_aoa}), even when under the same thrust limitations as the initial model.}

\subsection{\Ra{Comparison with Gramian-Based Metric}}~\label{subsec:gramian_comp}
\Ra{To compare our approach with existing literature, we consider the Gramian-based approach proposed in~\cite{gupta_ControllabilityGramianControl_2020} in the challenge of improving maneuverability in pitch angle.
The method in~\cite{gupta_ControllabilityGramianControl_2020} proceeds by maximizing the eigenvalues of $W_R$ corresponding to the eigenvectors that align most with the direction of interest.
In this case, this translates into an optimization problem to maximize the sums of the third and fourth eigenvalues of $W_R$ while minimizing the sum of the first two:
\begin{align}\label{eq:gramian_method}
    \max_{c,w,\Lambda,\ell,\Gamma} \quad & \lambda_3(W_R) + \lambda_4(W_R) - \lambda_1(W_R) - \lambda_2(W_R) \\ \nonumber
    \textrm{s.t.} \quad  &  3 \leq c \leq 7 \\ \nonumber 
    & 10 \leq w \leq 20 \\ \nonumber
    & 3 \leq \ell \leq 8 \\ \nonumber
    & 15 \leq \Lambda \leq 50 \\ \nonumber 
    & 10 \leq \Gamma \leq 12\\ \nonumber
    & \textrm{Static Margin} > 0.
\end{align}
The eigenvalues are chosen for maximization and minimization in the same manner as in~\cite{gupta_ControllabilityGramianControl_2020}, since we observe that the third and fourth eigenvectors of $W_R$ align most with the pitch angle, shown in Figure~\ref{fig:gram_eigvecs}.}
\begin{figure}[h]
    \centering
    \includegraphics[width=0.5\linewidth]{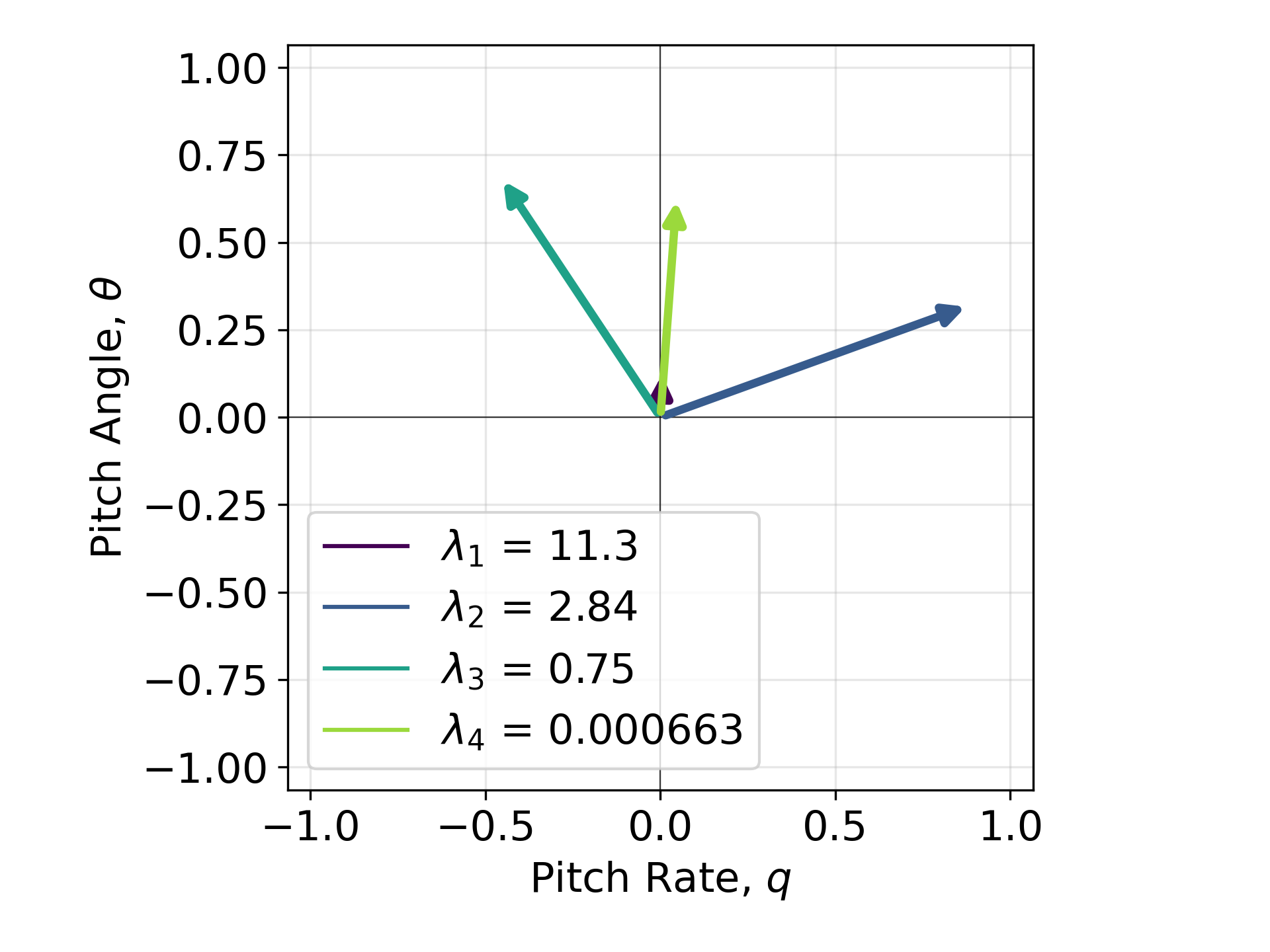}
    \caption{\Ra{Eigenvectors of the controllability Gramian projected onto states of pitch rate and pitch angle.
    The eigenvectors corresponding to the third and fourth eigenvalues have the largest components in pitch angle, so they are selected for maximization while the remaining two eigenvectors are minimized.}}
    \label{fig:gram_eigvecs}
\end{figure}
\Ra{As before, we consider the dynamics linearized around level-flight at $200$ m/s.}
\Ra{Solving optimization problem~\eqref{eq:gramian_method} yields $[c^*,\, w^*,\, \Lambda^*,\, \ell^*,\, \Gamma^*] = [4.6 \textrm{ m},\, 12.6\textrm{ m},\, 36.9\textrm{ deg},\, 4.6\textrm{ m},\, 10\textrm{ deg}]$.
In contrast, applying our method of reachable set projection \eqref{eq:dirmax_optimization} onto $\bv = [0,\, 0, \, 0, \, 1]^\top$} \Ra{yields $[c^*,\, w^*,\, \Lambda^*,\, \ell^*,\, \Gamma^*] = [7 \textrm{ m},\, 16.6\textrm{ m},\, 15\textrm{ deg},\, 8\textrm{ m},\, 10\textrm{ deg}]$ meters, mirroring the design of~\eqref{eq:volume_max_optimization}.}

\Ra{Comparing the nonlinear reference tracking performance using the procedure outlined in Subsection~\ref{subsec:nonlinear_tracking}, the tracking performance is shown in Figure~\ref{fig:gram_inf_comparison}.
Since the goal of the optimization is to improve maneuverability in pitch angle, we present the $\mathcal{L}^2$ norms of tracking errors and control costs in Table~\ref{tab:tracking_error_gramian_comparison} and transient performance metrics in Table~\ref{tab:transient_comparison}.
The $\Lp[2]$ norm of the thrust is not shown because the thrust signal is very similar between designs.
The settling time of our reachability-based design~\eqref{eq:dirmax_optimization} is faster than the settling time of the Gramian-based design, at 4.36 seconds compared to 9.65 seconds.
Furthermore, our design stabilizes to the reference using 92.1\% less elevator control effort and with 13.6\% less pitch angle tracking error.
Furthermore, from Figure~\ref{fig:gram_inf_comparison}, it is clear that the design based on~\eqref{eq:gramian_method} fails to converge to the reference in the first stage of the tracking problem, with its angle of attack, velocity, and pitch angle failing to reach their references within 10 seconds.
The main reason for this slowly decaying tracking error is that the Gramian-based design has its elevator surface completely saturated during the first phase of the tracking problem.
Since the Gramian-based design only considers the eigenvalues of the controllability Gramian, it fails to recognize when the dynamics of the aircraft limit its maneuverability.
On the other hand, considering the reachable set of the dynamics ensures that the optimal design can achieve large changes in state with minimal actuation.}

\begin{figure}[h]
    \centering
    \includegraphics[width=\linewidth]{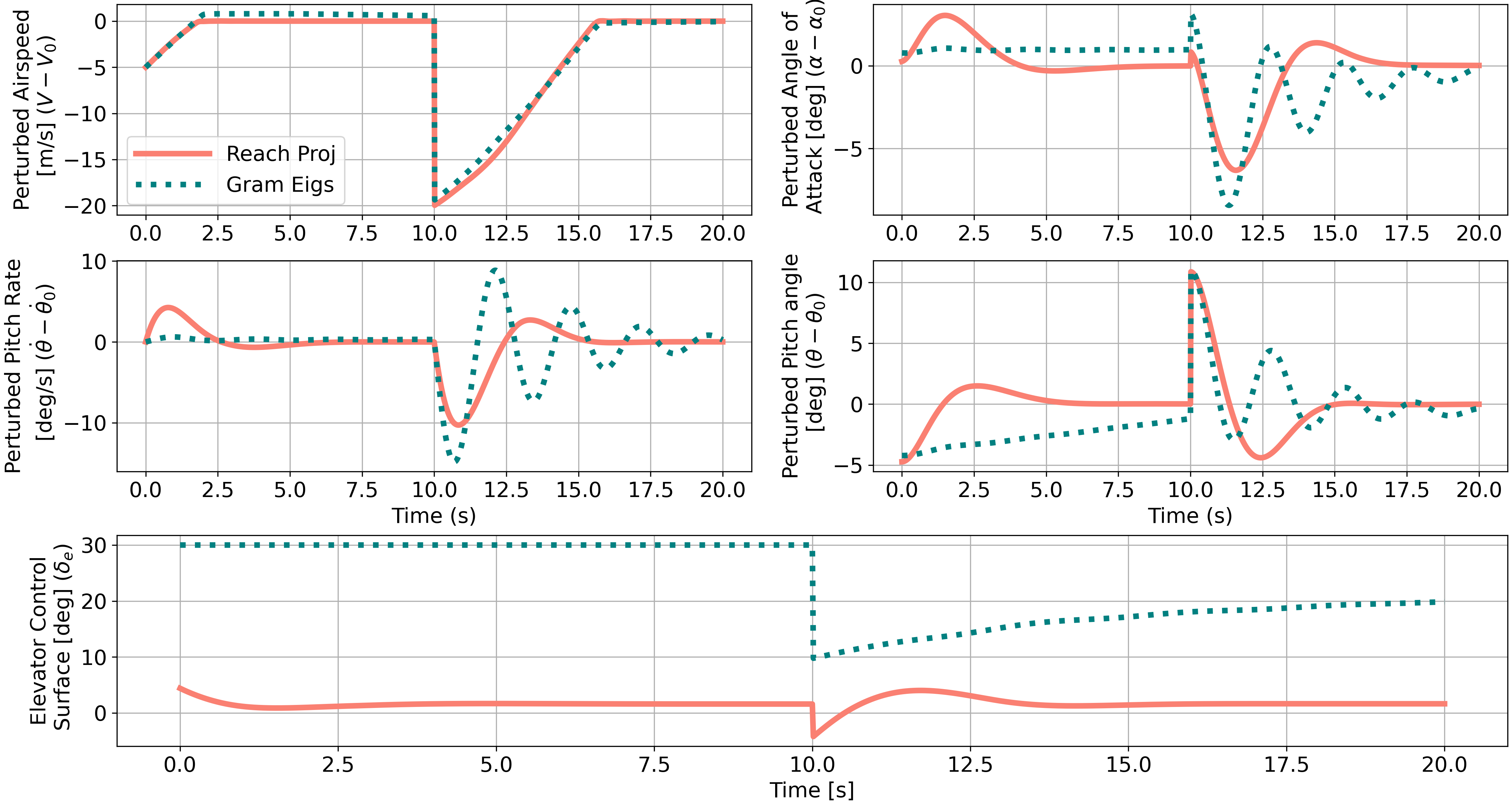}
    \caption{\Ra{Tracking error comparison of nonlinear reference tracking between the optimal designs based on Gramian eigenvalues (Gram Eigs) and reachability set projection (Reach Proj) with $\bv = [0,\, 0, \, 0, \, 1]^\top$.
    The reachability-based design uses less control effort to achieve tracking performance with less overshoot and faster convergence to the reference than the Gramian-based design.}}
    \label{fig:gram_inf_comparison}
\end{figure}

\begin{table}[ht]
    \centering
    \begin{tabular}{C{1.75cm}|C{2.25cm}|C{2.25cm}|C{2.25cm}|C{2.25cm}|C{2.25cm}}
    Design &  Velocity $\Lp[2]$ Error (\%~Improvement)& Angle of Attack $\Lp[2]$ Error (\%~Improvement) &  Pitch Rate $\Lp[2]$ Error (\%~Improvement) & Pitch Angle $\Lp[2]$ Error (\%~Improvement) & Elevator $\Lp[2]$ Cost (\%~Improvement) \\ \hline
    Gramian-Based Design~\eqref{eq:gramian_method}         &  \Ra{28.66            } & \Ra{0.17         } & \Ra{0.29         } & \Ra{0.22            } & \Ra{1.90   }       \\ \hline
    Opt. Design with~\eqref{eq:dirmax_optimization}        &  \Ra{30.12 ($-5.1\%$)}  & \Ra{0.15 (11.8\%)} & \Ra{0.22 (24.1\%)} & \Ra{0.19 ($13.6\%$)}  & \Ra{0.15 (92.1\%)}   \\
    \end{tabular}
    \caption{\Ra{Tracking errors and control costs for control of nonlinear dynamics between the Gramian-based design and reachability-based design. Percent improvement shows decrease in cost or error with respect to initial design (positive is reduction, negative
    is increase).}}
    \label{tab:tracking_error_gramian_comparison}
\end{table}

\begin{table}[ht]
    \centering
    \begin{tabular}{C{2cm}|C{2.5cm}|C{2.5cm}|C{2.5cm}}
    Design &  Rise Time (s) & Settling Time (s) & Overshoot (\%)\\ \hline
    Gramian-Based Design~\eqref{eq:gramian_method}         &  \Ra{0.98} & \Ra{9.65} & \Ra{40.01}\\ \hline
    Opt. Design with~\eqref{eq:dirmax_optimization}        &  \Ra{1.30 ($-32.7\%$)} & \Ra{4.36 (54.8\%)} & \Ra{40.53 ($-1.3\%$)}\\
    \end{tabular}
    \caption{\Ra{Pitch angle rise time, settling time, overshoot, and $\Lp[2]$ norm of elevator signal for Gramian-based design and reachability-based design in nonlinear reference tracking.}}
    \label{tab:transient_comparison}
\end{table}

\section{Conclusions}

In the design of next-generation aircraft, incorporating control concepts and resulting constraints into early stages of the design process could build critical checks for controllability and maneuverability into the first prototypes.
Rather than waiting until the design is finalized before checking its capabilities under feedback, we propose using reachability analysis to augment design optimization with control-oriented metrics.
Since computing reachable sets is notoriously expensive, we restrict to linear analysis to reduce the complexity of the problem while still gaining insight into the aircraft dynamics.
\Rb{Application of this approach to the BWB aircraft reveals that analysis of the linearized dynamics can uncover design changes that improve the aircraft's controlled performance, even when considering their nonlinear dynamics.}
\Ra{Furthermore, comparison with the existing approach of optimizing over the controllability Gramian~\cite{gupta_ControllabilityGramianControl_2020} reveals that using reachable sets can result in improved designs with more efficient control.}

For future work, further investigation of computing exact derivatives for metrics based on reachability could improve the speed and convergence of optimization for larger design problems.
Building on the analysis with the BWB, a more comprehensive design optimization problem \red{with lateral dynamics and a structural weight model} could reveal more insight into the design tradeoffs that can be made to improve maneuverability.
Furthermore, since the reachability analysis is based on the linear dynamics, optimizing the aircraft while computing reachable sets about multiple trim points could yield insight into larger \red{regions} of the flight envelope and ensure improved control performance throughout dynamic maneuvers.

\appendix
\section{Appendix: Smoothness of Trim}\label{appendix:trim_implicit}
Here, we show that the trim procedure is locally continuously differentiable.
For a given design vector $\bd = [c,\ w,\ \Ra{\Lambda,\ \ell,\ \Gamma}]^\top$, the trim problem is to find the trimmed values $\alpha_0,\deltaeo,\deltato$, and $\theta_0$ such that the $\dot{V} = \dot{\alpha} = \dot{Q} = 0$ and $\theta - \alpha = \gamma$ for a desired flight path angle $\gamma$ in~\eqref{eq:nonlinear_longitudinal_model}.
\red{We consider the set of design variables $\calD$ as defined using the variable bounds in Section~\ref{subsec:aero_data_collection}, and further impose $\theta \in [-30, 30]$ deg}. 
We denote $\calB = \red{[-5,\ 15] \times [-30,\ 30] \times [0,\ 1] \times [-30,\ 30]}$ and $\Omega = \calD \times \calB$.
Then, we define $\bF:\Omega \to \R^4$ as the right hand side of~\eqref{eq:nonlinear_longitudinal_model}.
Letting $\bx = [c,\ w,\ \Ra{\Lambda,\ \ell,\ \Gamma}]^\top \in \calD$ and $\by = [\alpha,\ \deltae,\ \deltat,\ \theta]^\top \in \calB$, we aim to show that there exists a continuously differentiable mapping between the design variables, $\bx$, and trimmed variables, $\by$.

\begin{proposition}
    Let $\bPhi: \calD \to \calB$ be a function satisfying $\bF(\bxi[0],\by_0) = \bolds{0}$, where $\bF:\Omega \to \R^4$ is defined as the right hand side of~\eqref{eq:nonlinear_longitudinal_model} and $\by_0 = \bPhi(\bxi[0])$.
    Then, $\bPhi$ is continuously differentiable in a neighborhood of $\bxi[0] = [c_0,\ w_0,\ \Ra{\Lambda_0,\ \ell_0,\ \Gamma_0}]^\top$.
\end{proposition}
\begin{proof}
    By assumption, $\bF(\bxi[0],\bPhi(\bxi[0])) = \bolds{0}$.
    If we can show that $\pder[\bF]{\by}$ is invertible in a neighborhood of $[\bxi[0]^\top,\by_0^\top]^\top$, then by the implicit function theorem, there exists a mapping $\bPhi:\tilde{\calD} \to \calB$ that solves the trim problem, where $\tilde{\calD} \subseteq \calD$ is a neighborhood of $\bxi[0]$.
    In~\eqref{eq:nonlinear_longitudinal_model}, the drag, lift, and moment terms are all computed using linear interpolations of sampled aerodynamic data points (Section~\ref{subsec:aero_data_collection}), so $\bF$ is continuously differentiable over $\Omega$.
    Taking the partial derivative with respect to $\by$, we get:
\begin{equation} \label{eq:partialFpartialy}
\red{\pder[\bF]{\by}} = \begin{pmatrix}
      \red{\dfrac{-F_t(\deltat)}{m}\sin(\alpha) - \dfrac{1}{m} \pder[D(\bx,\by)]{\alpha} + g \cos(\theta-\alpha)}
    & \red{\dfrac{-1}{m}\pder[D(\bx,\by)]{\deltae} }
    & \red{\dfrac{\cos(\alpha)}{C \cdot m} }
    & \red{- g\cos(\theta-\alpha) }\\
    
      \red{\dfrac{-F_t(\deltat)}{m \cdot V_0}\cos(\alpha) - \dfrac{1}{m \cdot V_0} \pder[L(\bx,\by)]{\alpha} + \dfrac{g}{V_0}\sin(\theta-\alpha) }
    & \red{\dfrac{-1}{m \cdot V_0} \pder[L(\bx,\by)]{\deltae} }
    & \red{\dfrac{-\sin(\alpha)}{C \cdot m \cdot V_0} }
    & \red{ - \dfrac{g}{V_0}\sin(\theta-\alpha)} \\
    
      \red{\pder[M(\bx,\by)]{\alpha} }
    & \red{\pder[M(\bx,\by)]{\deltae} }
    & \red{0 }
    & \red{0 }\\
    
    \red{-1} & \red{0} & \red{0} & \red{1}
\end{pmatrix},
\end{equation}
\red{where we have leveraged that the thrust force $F_t$ is independent of $\alpha,\deltae,$ and $\theta$.
We set $C = \pder[F_t]{\deltat}$ because the thrust force increases linearly with the command $\deltat \in [0,1]$ until reaching its max value of $C$.
We have also used the fact that derivatives of aerodynamic forces and moments with respect to $\theta$ are 0.}

\red{To show invertibility of $\pder[\bF]{\by}$, we will first add the fourth column with the first column to cancel the gravity terms in the first column:}
\begin{equation} \nonumber
    \begin{pmatrix}
          \red{\dfrac{-F_t(\deltat)}{m}\sin(\alpha) - \dfrac{1}{m} \pder[D(\bx,\by)]{\alpha}}
        & \red{\dfrac{-1}{m}\pder[D(\bx,\by)]{\deltae} }
        & \red{\dfrac{\cos(\alpha)}{C \cdot m} }
        & \red{- g\cos(\theta-\alpha) }\\
        
          \red{\dfrac{-F_t(\deltat)}{m \cdot V_0}\cos(\alpha) - \dfrac{1}{m \cdot V_0} \pder[L(\bx,\by)]{\alpha} }
        & \red{\dfrac{-1}{m \cdot V_0} \pder[L(\bx,\by)]{\deltae} }
        & \red{\dfrac{-\sin(\alpha)}{C \cdot m \cdot V_0} }
        & \red{ - \dfrac{g}{V_0}\sin(\theta-\alpha)} \\
        
          \red{\pder[M(\bx,\by)]{\alpha} }
        & \red{\pder[M(\bx,\by)]{\deltae} }
        & \red{0 }
        & \red{0 }\\
        
        \red{0} & \red{0} & \red{0} & \red{1}
\end{pmatrix}.
\end{equation}
\red{Multiplying row one by $m$ and row two by $mV_0$, we can rotate the two rows by angle $\alpha$ to obtain:}
\begin{equation} \label{eq:dfdy_reduced}
      \begin{pmatrix}
          \red{\pder[L]{\alpha} \sin(\alpha) - \pder[D]{\alpha} \cos(\alpha)}
        & \red{\pder[L]{\deltae} \sin(\alpha) - \pder[D]{\deltae} \cos(\alpha) }
        & \red{\frac{1}{C}}
        & \red{-mg\cos(\theta)}\\
        
          \red{-F_t - \pder[L]{\alpha} \cos(\alpha) - \pder[D]{\alpha}\sin(\alpha)}
        & \red{-\pder[L]{\deltae} \cos(\alpha) - \pder[D]{\deltae}\sin(\alpha)}
        & \red{0}
        & \red{ - mg\sin(\theta)} \\
        
          \red{\pder[M]{\alpha} }
        & \red{\pder[M]{\deltae} }
        & \red{0 }
        & \red{0 }\\
        
        \red{0} & \red{0} & \red{0} & \red{1}
\end{pmatrix}.
\end{equation}
\red{Since these operations consist of linear combinations of rows and columns, the rank of\eqref{eq:dfdy_reduced} and~\eqref{eq:partialFpartialy} are the same.
Taking the determinant of~\eqref{eq:dfdy_reduced}, we get:}
\begin{equation}\label{eq:determinant_trim}
    \red{\frac{1}{C} \left[ -\pder[M]{\deltae} mg \sin(\theta) + \pder[M]{\deltae}\left(-F_t - \pder[L]{\alpha} \cos(\alpha) - \pder[D]{\alpha}\sin(\alpha) - mg \sin(\theta)\right) + \pder[M]{\alpha}\left(\pder[L]{\deltae}\cos(\alpha) + \pder[D]{\deltae}\sin(\alpha)\right) \right]}
\end{equation}
\red{If~\eqref{eq:determinant_trim} is nonzero for $\bxi[0]$ and $\by_0$, then the proof will be done.
Simplifying~\eqref{eq:determinant_trim}, we can ignore the constant $\frac{1}{C}$ and rearrange terms:
\begin{align}\nonumber
    &-\pder[M]{\deltae}\left(F_t + \ol{N}_\alpha \right) + \pder[M]{\alpha}N_{\deltae} \\
    = & -\pder[M]{\deltae}\left(F_t + \ol{N}_{\alpha} \right)(1-\rho) \label{eq:simplified_determinant}
\end{align}
where $\ol{N}_\alpha = \pder[L]{\alpha}\cos(\alpha) + \pder[D]{\alpha}\sin(\alpha)$, $N_{\deltae} = \pder[L]{\deltae} \cos(\alpha) + \pder[D]{\deltae} \sin(\alpha)$, and $\rho = \frac{N_{\delta_\mathrm{e}}\,\partial M/\partial\alpha}{(F_t+\overline{N}_\alpha)\,\partial M/\partial\delta_\mathrm{e}}$.
We know that $\pder[M]{\deltae} \neq 0$ since the elevator surface is designed to control the pitch of the aircraft and $F_t \geq 0$ because thrust only acts one direction.
Moreover, $\ol{N}_\alpha >0$ because we consider $\alpha \in [-5,15]$ degrees, so $\pder[L]{\alpha}\cos(\alpha)$ dominates and is positive because lift increases with $\alpha$.
Lastly, we can check the value of $\rho$ throughout our entire region of collected data to find that it has a maximum magnitude of 0.013, so~\eqref{eq:simplified_determinant} is nonzero.
Thus, $\pder[\bF]{\by}$ is invertible and we have completed the proof.}

\end{proof}

\section*{Funding Sources}
S.N., N.O., and B.K. would like to acknowledge the Multidisciplinary Science and Technology Center of the Aerospace Systems Directorate, Air Force Research Laboratory for funding and supporting this effort through the Collaborative Center for Design and Research of Interdisciplinary Systems program.
J.C.'s work was partially supported by AFOSR Award FA9550-23-1-0740.

\section*{Acknowledgments}
Distribution A: Approved for public release; distribution unlimited. Case no. AFRL-2026-1972.
The authors would like to thank John Hwang, Gregory Reich, David Doman, Michael Bolender, \red{and Fey Liu} for their contributions to this project through many discussions.

\bibliography{MyRefs}

\end{document}